\documentclass[10pt]{article}

\usepackage{amssymb}
\usepackage{amsmath}
\usepackage{theorem}
\usepackage{epsfig}
\usepackage{verbatim}
\usepackage{graphicx}
\usepackage{subfigure}
\usepackage{cancel}
\usepackage{epstopdf}
\usepackage{longtable}

\usepackage{color}

\newtheorem{theorem}{Theorem}[section]
\newtheorem{lemma}[theorem]{Lemma}

\newcommand{\T}{\mathbb{T}}

\newcommand{\sign}{{\rm sign}\thinspace}

\newcommand{\ep}{\varepsilon}

\newcommand{\pa}{\partial}

\newcommand{\alx}{x}
\newcommand{\bey}{y}

\newenvironment{proof}{\begin{trivlist} \item[] {\em Proof:}}{\hfill $\Box$
                       \end{trivlist}}

\newenvironment{proofthm}[1]{\begin{trivlist} \item[] {\em Proof of Theorem \ref{#1}:}}{\hfill $\Box$
                       \end{trivlist}}

\makeatletter
\renewcommand*\l@section{\@dottedtocline{1}{0em}{1.5em}}
\renewcommand*\l@subsection{\@dottedtocline{2}{1.5em}{2.3em}}
\renewcommand*\l@subsubsection{\@dottedtocline{3}{3.8em}{3.7em}}
\makeatother

\numberwithin{equation}{section}

\allowdisplaybreaks[4]

\begin{document}

\title{A note on stability shifting for the Muskat problem II: \\ Stable to Unstable and back to Stable}

\author{Diego C\'ordoba, Javier G\'omez-Serrano, and Andrej Zlato\v{s}}

\maketitle

\begin{abstract}
In this note, we show that there exist solutions of the Muskat problem which shift stability regimes in the following sense: they start stable, then become unstable, and finally return back to the stable regime. 
This proves existence of double stability shifting in the direction opposite to the one shown in \cite{Cordoba-GomezSerrano-Zlatos:stability-shifting-muskat}.

\vskip 0.3cm
\textit{Keywords: Muskat problem, interface, incompressible fluid, porous media, Rayleigh-Taylor}

\end{abstract}


\section{Introduction}

In this paper, we study two incompressible fluids with the same viscosity but different densities, $\rho^{+}$ and $\rho^{-}$, evolving in a two dimensional porous medium with constant permeability $\kappa$. The velocity $v$ is determined by Darcy's law
\begin{equation}\label{IIIdarcy}
\mu\frac{v}{\kappa}=-\nabla p-g\left(\begin{array}{cc}0\\ \rho\end{array}\right),
\end{equation}
where $p$ is the pressure, $\mu>0$ viscosity, and $g > 0$ gravitational acceleration. In addition, $v$ is incompressible:
\begin{equation}\label{IIIincom}
\nabla\cdot v=0.
\end{equation}
By rescaling properly, we can assume $\kappa = \mu=g=1$. The fluids also satisfy the conservation of mass equation
\begin{equation}\label{IIIconser}
\partial_t\rho+v\cdot\nabla\rho=0.
\end{equation}

This is known as the Muskat problem \cite{Muskat:porous-media}. We denote by $\Omega^{+}$ the region occupied by the fluid with density $\rho^{+}$ and by $\Omega^{-}$ the region occupied by the fluid with density $\rho^{-} \neq \rho^{+}$. The point $(0,\infty)$ belongs to $\Omega^{+}$, whereas the point $(0,-\infty)$ belongs to $\Omega^{-}$. All quantities with superindex $\pm$ will refer to $\Omega^{\pm}$ respectively. The interface between both fluids at any time $t$ is a planar curve $z(\cdot,t)$. 
We will work in the setting of horizontally periodic interfaces, although our results can be extended to the flat at infinity case.

A quantity that will play a major role in this paper is the Rayleigh-Taylor condition, which is defined as
\begin{align*}
RT(\alx,t)=-\left[ \nabla p^{-}(z(\alx,t))-\nabla p^{+}(z(\alx,t)) \right]\cdot\partial_\alx^\bot z(\alx,t),
\end{align*}
where we use the convention $(u,v)^{\perp} = (-v,u)$. If $RT(\alx,t)>0$ for all $\alx\in\mathbb{R}$, then we say that the curve is in the Rayleigh-Taylor stable regime at time $t$, and if $RT(\alx,t) \leq 0$ for some $\alx\in\mathbb{R}$, then we  say that the curve is in the Rayleigh-Taylor unstable regime. 

One can rewrite the system \eqref{IIIdarcy}--\eqref{IIIconser} in terms of the curve $z=(z^1,z^2)$, obtaining
\begin{align}\label{muskatinterface}
\partial_{t} z(\alx,t) = \frac{\rho^{-} - \rho^{+}}{2\pi} P.V. \int_\mathbb{R} \frac{z^1(\alx,t) - z^1(\bey,t)}{|z(\alx,t) - z(\bey,t)|^{2}}(\partial_{\alx}z(\alx,t) - \partial_{\bey}z(\bey,t)) d\bey.
\end{align}
In the horizontally periodic case with $z(\alx+2\pi,t)=z(\alx,t)+(2\pi,0)$,  the formula
\begin{align*}
\frac12\cot\frac{y}{2} = \frac{1}{y} + \sum_{n=1}^{\infty} \frac{2y}{y^{2} - (2\pi n)^{2}}
\end{align*}
can be used to show \cite{Castro-Cordoba-Fefferman-Gancedo-LopezFernandez:rayleigh-taylor-breakdown} that the velocity  satisfies
\begin{align}\label{muskatinterfaceperiodic}
\partial_{t} z(\alx,t) = \frac{\rho^{-} - \rho^{+}}{4\pi} \int_\mathbb{T} \frac{\sin(z^1(\alx,t) - z^1(\bey,t))(\partial_{\alx}z(\alx,t) - \partial_{\bey}z(\bey,t))}{\cosh(z^{2}(\alx,t) - z^{2}(\bey,t)) - \cos(z^{1}(\alx,t) - z^{1}(\bey,t))} d\bey.
\end{align}
  A simple calculation of the Rayleigh-Taylor condition in terms of $z$ yields
\begin{align*}
 RT(\alx,t) = (\rho^{-} - \rho^{+})\partial_{\alx} z^{1}(\alx,t).
\end{align*}
When the interface is a graph, parametrized as $z(\alx,t)=(\alx,f(\alx,t))$, equation \eqref{muskatinterface} becomes
\begin{align}\label{muskatinterfacegraph}
 \partial_t f(x,t) = \frac{\rho^{-} - \rho^{+}}{4\pi} \int_\mathbb{T} \frac{\sin(x-y)(\partial_{x}f(x,t) - \partial_{y}f(y,t))}{\cosh(f(x,t) - f(y,t)) - \cos(x-y)} d y
\end{align}
and the Rayleigh-Taylor condition simplifies to
\begin{align*}
RT(\alx,t) =\rho^{-}-\rho^{+}.
\end{align*}
The curve is now in the RT stable regime whenever $\rho^{+}<\rho^{-}$, that is, the denser fluid is at the bottom. From now on, we assume that $\rho^{-} - \rho^{+} = 4\pi$, which can be done after an appropriate scaling in time.

The Muskat problem  has been studied in many works. A proof of local existence of classical solutions in the Rayleigh-Taylor stable regime in $H^{3}$ and ill-posedness in the unstable regime appears in \cite{Cordoba-Gancedo:contour-dynamics-3d-porous-medium}. See also \cite{Constantin-Gancedo-Shvydkoy-Vicol:global-regularity-muskat-finite-slope} for an improvement on the regularity (to $W^{2,p}$ spaces). In the one phase case (i.e. when one of the densities and permeabilities is zero) local existence in $H^{2}$ was proved in \cite{Cheng-GraneroBelinchon-Shkoller:well-posedness-h2-muskat}.

A maximum principle for $\| \pa_xf(\cdot,t)\|_{L^\infty}$ can be found in \cite{Cordoba-Gancedo:maximum-principle-muskat}. Moreover, the authors showed in \cite{Cordoba-Gancedo:maximum-principle-muskat} that if $\|\partial_{x} f_0\|_{L^\infty}<1$, then $\|\partial_{x} f(\cdot,t)\|_{L^\infty}\le \|\partial_{x} f_0\|_{L^\infty}$ for all $t>0$. Further work has shown existence of finite time turning \cite{Castro-Cordoba-Fefferman-Gancedo-LopezFernandez:rayleigh-taylor-breakdown} (i.e., the curve ceases to be a graph in finite time and the Rayleigh-Taylor condition changes sign to negative somewhere along the curve). The gap between these two results (i.e., the question whether the constant 1 is sharp or not for guaranteeing global existence) is still an open question, and there is numerical evidence of data with $\|\pa_x f_{0}\|_{L^{\infty}} = 50$ which turns over \cite{Cordoba-GomezSerrano-Zlatos:stability-shifting-muskat}.



As was demonstrated in \cite{Castro-Cordoba-Fefferman-Gancedo:breakdown-muskat}, the curve may lose regularity after shifting from the stable regime to the unstable regime.  However, the possibility of it recoiling and returning to the stable regime has not been excluded.  The occurrence of this phenomenon is the main result of this note, Theorem \ref{theoremshifting}. 
(In Theorem \ref{T.1.1} we also extend this to a proof of existence of the quadruple stability shift scenario unstable $\to$ stable $\to$ unstable $\to$ stable $\to$ unstable.)  In \cite{Cordoba-GomezSerrano-Zlatos:stability-shifting-muskat} we showed that there exist curves which undergo the unstable $\rightarrow$ stable $\rightarrow$ unstable  transition, so this settles the question about existence of double stability shift scenarios in both directions.

 More general models, which take into account finite depth or non-constant permeability, and which also exhibit (single) turning were studied in \cite{Berselli-Cordoba-GraneroBelinchon:local-solvability-singularities-muskat,GomezSerrano-GraneroBelinchon:turning-muskat-computer-assisted}.  The estimates in \cite{GomezSerrano-GraneroBelinchon:turning-muskat-computer-assisted} were carried out by rigorous computer-assisted methods, as opposed to the traditional pencil and paper ones in \cite{Berselli-Cordoba-GraneroBelinchon:local-solvability-singularities-muskat}. 
 
 Concerning global existence, the first proof for small initial data was carried out in \cite{Siegel-Caflisch-Howison:global-existence-muskat} in the case where the fluids have different viscosities and the same densities (see also \cite{Cordoba-Gancedo:contour-dynamics-3d-porous-medium} for the setting of the present paper --- different densities and the same viscosities --- and also \cite{Cheng-GraneroBelinchon-Shkoller:well-posedness-h2-muskat} for the general case). 
Global existence for medium-sized initial data was established in \cite{Constantin-Cordoba-Gancedo-Strain:global-existence-muskat,Constantin-Cordoba-Gancedo-RodriguezPiazza-Strain:muskat-global-2d-3d}. In the case where surface tension is taken into account, global existence was shown in \cite{Escher-Matioc:parabolicity-muskat,Friedman-Tao:nonlinear-stability-muskat-capillary-pressure}. Global existence for the confined case was treated in \cite{GraneroBelinchon:global-existence-confined-muskat}. A blow-up criterion was found in \cite{Constantin-Gancedo-Shvydkoy-Vicol:global-regularity-muskat-finite-slope}.

Recent advances in computing power have made  possible rigorous computer-assisted proofs.  Of course,
floating-point operations can result in numerical errors, and
we will employ interval arithmetics to deal with this issue.  Instead of working with arbitrary real numbers, we perform computations over intervals which have representable numbers as endpoints. On these objects, an arithmetic is defined in such a way that we are guaranteed that for every $x \in X, y \in Y$
\begin{align*}
x \star y \in X \star Y,
\end{align*}
for any operation $\star$. For example,
\begin{align*}
[\underline{x},\overline{x}] + [\underline{y},\overline{y}] & = [\underline{x} + \underline{y}, \overline{x} + \overline{y}] \\
[\underline{x},\overline{x}] \times [\underline{y},\overline{y}] & = [\min\{\underline{x}\underline{y},\underline{x}\overline{y},\overline{x}\underline{y},\overline{x}\overline{y}\},\max\{\underline{x}\underline{y},\underline{x}\overline{y},\overline{x}\underline{y},\overline{x}\overline{y}\}].
\end{align*}
We can also define the interval version of a function $f(X)$ as an interval $I$ that satisfies that for every $x \in X$, $f(x) \in I$. Rigorous computation of integrals has been theoretically developed since the seminal works of Moore and many others (see \cite{Berz-Makino:high-dimensional-quadrature,Kramer-Wedner:adaptive-gauss-legendre-verified-computation,Lang:multidimensional-verified-gaussian-quadrature,Moore-Bierbaum:methods-applications-interval-analysis,Tucker:validated-numerics-book} for just a small sample).
For readability purposes, instead of writing the intervals as, for instance, $[123456,123789]$, we will sometimes instead refer to them as $123^{456}_{789}$.

This note is organized as follows.  In Section \ref{SectionTheorem} we prove Theorems \ref{theoremshifting} and \ref{T.1.1}, and in Section \ref{SectionTechnicalDetails} we provide technical details regarding the performance and implementation of the computations. The appendix contains a detailed derivation and enumeration of all the necessary integrals which have to be rigorously computed, their enclosures, and the performance of the computations.

\section{The main result}
\label{SectionTheorem}

The following theorem is the main result of this paper (see also Theorem \ref{T.1.1} below).

\begin{theorem}\label{theoremshifting}
There exist $T>\gamma>0$  and a spatially analytic solution $z$ to \eqref{muskatinterfaceperiodic} on the time interval $[-T,T]$  such that $z(\cdot,t)$ is a graph of a smooth function of $x$ when $|t|\in[T-\gamma,T]$
(i.e., $z$ is in the stable regime near $t=\pm T$)  but $z(\cdot,t)$ is not a graph of a function of $x$ when $|t|\le \gamma$ (i.e., $z$ is in the unstable regime near $t=0$). 
\end{theorem}


The intuition behind this result comes from the numerical experiments which were started in \cite{Cordoba-GomezSerrano-Zlatos:stability-shifting-muskat}. These suggested  existence of curves which are (barely) in the unstable regime, 
and such that the evolution both forward and backwards in time transports them into the stable regime. (We note that neither the velocity nor any other quantity was observed to become degenerate  in these experiments.)
The following lemma constructs a family of such curves.

\begin{lemma}
\label{lemacomputer}
Let $\ep\ge 0$ and consider the initial curve $z_{\ep}(\alx,0)=(z^{1}_{\ep}(\alx,0),z^{2}_{\ep}(\alx,0))$, with
\begin{align*}
 z^{1}_{\ep}(\alx,0) & = \alx - \sin(\alx) - \ep \sin(\alx), \\
z^{2}_{\ep}(\alx,0) & = A(\ep) \sin(2\alx).
\end{align*}
\begin{enumerate}
\item For any $\ep\in[0,10^{-6}]$, there exists $A(\ep)\in(1.08050, 1.08055)$ such that if $z_\ep$ solves \eqref{muskatinterfaceperiodic} with  initial data $z_{\ep}(\alx,0)$, then
\begin{align*}
 \pa_{t} \pa_{\alx} z^{1}_{\ep}(0,0) = 0.
\end{align*}
\item  There are $T,C>0$ such that for any $\ep\in[0,10^{-6}]$ and $A(\ep)$ from 1., there is a unique analytic solution $z_{\ep}$ of \eqref{muskatinterfaceperiodic} on the time interval $(-T,T)$ with initial data $z_{\ep}(\alx,0)$, and it satisfies
\begin{equation} \label{2.1}
\pa_{tt} \pa_{\alx} z^{1}_{\ep}(0,0)\ge 30
\end{equation}
as well as
\begin{align} \label {2.2}
 |\pa_{t}\pa^{3}_{x}z_{\ep}^{1}(x,t)| + |\pa_{t}^{2}\pa^{2}_{x}z_{\ep}^{1}(x,t)| + | \pa_{t}^{3} \pa_{x} z_{\ep}^{1}(x,t)| \leq C
\end{align}
for each $(x,t)\in\mathbb{T}\times(-T,T)$.
\end{enumerate}
\end{lemma}


\begin{proof}
The proofs of 1.~and \eqref{2.1} are computer-assisted, and the codes can be found in the supplementary material. 

Let us start with 1. Since $\pa_{t} \pa_{\alx} z^{1}_{\ep}(0,0)$ (i.e., the spatial derivative of the first coordinate of the the right-hand side of \eqref{muskatinterfaceperiodic} at $(x,t)=(0,0)$) is a continuous function of $A(\ep)$, it suffices to show that the signs of $\pa_{t} \pa_{\alx} z^{1}_{\ep}(0,0)$ for $A(\ep) = 1.08050$ and for $A(\ep) = 1.08055$ are different for each $\ep\in[0,10^{-6}]$. This holds because for each such $\ep$ we obtain the  bounds
\begin{align} \label{2.4}
\pa_{t} \pa_{\alx} z^{1}_{\ep}(0,0) \in 0.000^{01}_{27}
& \qquad \text{ for } A(\ep) = 1.08050,  \\
\pa_{t} \pa_{\alx} z^{1}_{\ep}(0,0) \in  -0.000^{28}_{02}
& \qquad \text{ for } A(\ep) = 1.08055. \notag
\end{align}

Existence and uniqueness of the solution $z_\ep$ in 2.~follows directly from the proof of Theorem 5.1 in \cite{Castro-Cordoba-Fefferman-Gancedo-LopezFernandez:rayleigh-taylor-breakdown}, which proves local well-posedness for   \eqref{muskatinterfaceperiodic} in the class of analytic functions of $x$.  The time $T>0$ 
is then uniform in all small $\ep$ (and $\sup_{|t|<T}\|\pa_x^k z_\ep(\cdot,t)\|_{L^\infty}$ is also uniformly bounded for each $k$) because the same is true for all the estimates in that proof.  

Then \eqref{2.1} follows by
taking $A(\ep) = [1.08050,1.08055]$ (the full interval, since we do not know $A(\ep)$ explicitly) and propagating this interval in the relevant computations. Specifically, we obtain 
\begin{align*}
\pa_{tt} \pa_{\alx} z^{1}_{\ep}(0,0) \in [38.706,48.787]. 
\end{align*}

The proof of Theorem 5.1 in \cite{Castro-Cordoba-Fefferman-Gancedo-LopezFernandez:rayleigh-taylor-breakdown} shows that  the chord-arc constant
\[ 
\sup_{x,y\in\mathbb T \,\&\, y\neq 0} \frac{|y|}{|z_\ep(x,t)-z_\ep(x-y,t)|}
\]
(where $\mathbb T=[-\pi,\pi]$ with $\pm\pi$ identified)
is bounded uniformly in all $\ep,t$ under consideration (provided $T>0$ is small enough).  Thus there is $D<\infty$ such that
\begin{align*}
 \left|\frac{1}{\cosh(z_{\ep}^2(x,t) - z_{\ep}^2(x-y,t)) - \cos(z_{\ep}^1(x,t) - z_{\ep}^1(x-y,t))} \right| \leq \frac D{y^2}
\end{align*}
for all these $\ep,t$ and all $x,y\in\mathbb T$.
This allows us to obtain \eqref{2.2} by brute force, differentiating and estimating all the resulting terms separately. The most singular term in $\pa_{t}\pa_{x}^3z_{\ep}^{1}(x,t)$ is given by
\begin{align*}
 \int_{\T} \frac{\sin(z_{\ep}^1(x,t) - z_{\ep}^1(x-y,t))(\pa_{x}^{4}z_{\ep}^{1}(x,t) - \pa^{4}_{x} z_{\ep}^{1}(x-y,t))}{\cosh(z_{\ep}^2(x,t) - z_{\ep}^2(x-y,t)) - \cos(z_{\ep}^1(x,t) - z_{\ep}^1(x-y,t))}dy \leq  2\pi D \| \pa_{x}^{5}z_{\ep}^{1}(\cdot,t)\|_{L^{\infty}} \leq C
\end{align*}
for some $C$ which is uniform in $\ep$ due to the above-mentioned uniform bounds on $\|\pa_x^k z_\ep(\cdot,t)\|_{L^\infty}$.
 Analogously, the most singular term in $\pa_{t}^{2} \pa_{x}^{2} z_{\ep}^{1}(x,t)$ is given by
\begin{align*}
 \int_{\T} \frac{\sin(z_{\ep}^1(x,t) - z_{\ep}^1(x-y,0))(\pa_{t} \pa_{x}^{3}z_{\ep}^{1}(x,t) - \pa_{t} \pa_{x}^{3} z_{\ep}^{1}(x-y,t))}{\cosh(z_{\ep}^2(x,t) - z_{\ep}^2(x-y,t)) - \cos(z_{\ep}^1(x,t) - z_{\ep}^1(x-y,t))}dy 
 & \leq 2\pi D\| \pa_{t} \pa_{x}^{4}z_{\ep}^{1}(\cdot,t)\|_{L^{\infty}}, 
\end{align*}
and the last term can be bounded by a uniform $C$ in the same way as $\pa_{t}\pa_{x}^3z_{\ep}^{1}(x,t)$.
 Finally, the most singular term in $\pa_{t}^{3}\pa_{x} z_\ep^{1}(x,t)$ is
\begin{align*}
 \int_{\T} \frac{\sin(z_{\ep}^1(x,t) - z_{\ep}^1(x-y,t))(\pa_{t}^{2} \pa_{x}^{2}z_{\ep}^{1}(x,t) - \pa_{t}^{2} \pa_{x}^{2} z_{\ep}^{1}(x-y,t))}{\cosh(z_{\ep}^2(x,t) - z_{\ep}^2(x-y,t)) - \cos(z_{\ep}^1(x,t) - z_{\ep}^1(x-y,t))}dy 
 & \leq 2\pi D\| \pa^{2}_{t} \pa_{x}^{3}z_{\ep}^{1}(\cdot,t)\|_{L^{\infty}}, 
\end{align*}
which is bounded by a uniform $C$ in the same way as $\pa_{t}^{2} \pa_{x}^{2} z_{\ep}^{1}(x,t)$ (with the bound this time involving the uniformly bounded quantity $\| \pa_{x}^{7}z_{\ep}^{1}(\cdot,t)\|_{L^{\infty}}$).
\end{proof}


\begin{proofthm}{theoremshifting}
Let $\ep,T,C$ be from the lemma (assume also $C\ge 1$), and note that we also have
\begin{align}
\label{derivadauniforme}
|\pa_{t} \pa_{x} z^{1}_{\ep}(x,t)| \le C
\end{align}
for $\ep\in[0,10^{-6}]$ and $(x,t)\in\mathbb T\times (-T,T)$, which is obtained  as the estimate on $\pa_{t}\pa_{x}^3z_{\ep}^{1}(x,t)$.  This means that for any small enough $\ep>0$ and any $|t|\le\sqrt\ep$ and $|x|\in [2C^{1/2}\ep^{1/4},\pi]$ we have
$\pa_{x} z^{1}_{\ep}(x,t)>0$.  

Next let $|t|\le\sqrt\ep$ and $|x|\le2C^{1/2}\ep^{1/4}$. Then there are $|x^{\sharp}|, |x^{\sharp \sharp}| \le |x|$ and $|t^{\sharp}| \le \sqrt \ep$ such that
\begin{align*}
\pa_{x} z^{1}_{\ep}(x,t)  = & \pa_{x} z^{1}_{\ep}(x,0) + t \pa_{t} \pa_{x} z^{1}_{\ep}(x,0) + \frac12 t^{2} \pa_{t}^{2} \pa_{x} z^{1}_{\ep}(x,0) + \frac{1}{6}t^{3} \pa_{t}^{3} \pa_{x} z^{1}_{\ep}(x,t^{\sharp}) \\
 = & -\ep\cos(x) + [1-\cos(x)]  + t \left[ \pa_{t} \pa_{x} z^{1}_{\ep}(0,0) + x\pa_{t} \pa_{x}^{2} z^{1}_{\ep}(0,0) + \frac12 x^{2}\pa_{t} \pa_{x}^{3} z^{1}_{\ep}(x^{\sharp},0) \right] \\
& + \frac12 t^{2} \left[ \pa_{t}^{2} \pa_{x} z^{1}_{\ep}(0,0) + x\pa_{t}^{2} \pa_{x}^{2} z^{1}_{\ep}(x^{\sharp \sharp},0) \right] + \frac{1}{6}t^{3} \pa_{t}^{3} \pa_{x} z^{1}_{\ep}(x,t^{\sharp}) \\
 \ge & -\ep  + x^2 \left(\frac 14-\frac C2|t| \right) + t^2 \left(15 -\frac C2|x| - \frac C6 |t| \right),
\end{align*}
where we  used the estimates from Lemma \ref{lemacomputer} and also  that $\pa_{t}\pa_{x}^2z^{1}_{\ep}(0,0) = 0$ by oddness of $z_\ep$.   This, together with $|t|= \sqrt\ep$ and $|x|\le 2C^{1/2}\ep^{1/4}$, shows that if $\ep>0$ is small enough, then $\pa_{x} z^{1}_{\ep}(x,t)>0$ for all $|t|\in[\frac 12\sqrt\ep,\sqrt\ep]$ and $|x|\le2C^{1/2}\ep^{1/4}$.  The theorem now follows with $z=z_\ep$, $T=\sqrt\ep$, and $\gamma=\frac \ep {2C}$, provided $\ep>0$ is  small enough (here we also used \eqref{derivadauniforme} and $\pa_{x} z^{1}_{\ep}(0,0)=-\ep$).
\end{proofthm}

We next show that our approach allows for the proof of existence of solutions which exhibit even more complicated stability shifting.  We will construct a solution with an unstable $\to$ stable $\to$ unstable $\to$ stable $\to$ unstable  stability regime profile.

We start by noticing that it suffices to consider solutions to \eqref{muskatinterface} with periodicity of the form $z(x+8N\pi,t)=z(x,t)+(8N\pi,0)$ for some integer $N\ge 1$, because then  $\tilde z(x,t)=\frac 1N z(Nx,Nt)$ also solves \eqref{muskatinterface} and  $\tilde z(x+2\pi,t)=\tilde z(x,t)+(2\pi,0)$.  Our initial data will be a perturbation of the $8N\pi$-periodic extension of the odd function
\[
z(x,0)=\bar z_{A(0)}(x)\chi_{[0,N\pi]}(|x|) + \bar z_{1.08055}(x)\chi_{(N\pi,3N\pi]}(|x|) + \bar z_{1.08050}(x)\chi_{(3N\pi,4N\pi]}(|x|),
\]
with $\bar z_B(x)=(x-\sin x,B\sin(2x))$ and $A(0)\in(1.08050, 1.08055)$ from Lemma \ref{lemacomputer}.  If $N$ is large, the estimates from the lemma and its proof show that at time $t=0$, the corresponding solution wants to make the shifts unstable $\to$ stable $\to$ unstable at $x=0$, stable $\to$ unstable at $|x|=2N\pi$, and stable $\to$ unstable at $|x|=4N\pi$.  An appropriate  perturbation of this initial data, which makes the unstable phase of the first shift last a positive length of time, delays the second shift, and brings the third shift forward in time would then achieve our goal.  We will also need this perturbation to resolve some other issues.  Specifically, the initial condition must be analytic so that we can solve the PDE both forward and  backward in time, and the solution must remain stable near $x=2n\pi$ for any integer $|n|\in(0,2N)\setminus\{N\}$ (note that the tangent to $z(x,0)$ is vertical at these points). 

For any large $N$ we therefore let
\[
B_{N,A}(x)= [A+(1.08055-A)\phi(|x|-N\pi)] \chi_{[0,3N\pi]}(|x|) + [1.08050+0.00005\phi(3N\pi+1-|x|)] \chi_{(3N\pi,4N\pi]}(|x|),
\]
with $A\in[1.08050, 1.08055]$ and $0\le\phi\le 1$ smooth such that $\phi(y)=0$ for $y\le 0$ and $\phi(y)=1$ for $y\ge 1$, and we extend $B_N$ to $\mathbb R$ periodically (with period $8N$).  Next we let $\delta_y$ be the delta function at $y\in\mathbb R$, and define the $8N\pi$-periodic odd functions
\[
\bar z_{N,A}(x)= (x-\sin x,B_{N,A}(x)\sin(2x)),
\]
and
\[
z_{N,A,r,a,c}(\cdot,0)=P_r* \bar z_{N,A} -  P_1*(a\beta_{N,0} + c \beta_{N,2N\pi} + c\beta_{N,-2N\pi}+ c \beta_{N,4N\pi},0),
\]
with $P_r(x)=\frac 1\pi \frac r{x^2+r^2}$ the Poisson kernel for the half-plane (note that $P_r*{\rm Id}_{\mathbb R}={\rm Id}_{\mathbb R}$) and $\beta_{N,y}(x)=\beta_N(x-y)$, where $\beta_N$ is the (unique and $8N\pi$-periodic) primitive of
\[
\sum_{n\in\mathbb Z} \delta_{8N\pi n} -\frac 1{8N\pi}
\]
which has $\int_{-4N\pi}^{4N\pi} \beta_{N}(x)dx=0$.  This and smoothness of $\phi$ means $z_{N,A,r,a,c}(\cdot,0)$ can be extended analytically to the strip $S_{r}=\mathbb R\times[-r,r]$ and this extension satisfies for each $k\ge 0$, 
\[
\sup_{N\ge 1 \,\&\, A\in[1.08050, 1.08055] \,\&\, r,a,c\in[0,1/2] \,\&\, |\zeta|\le r} \|\pa_x^k z_{N,A,r,a,c}(\cdot+i\zeta,0)\|_{L^\infty}<\infty.
\]
The proof of Theorem 5.1 in \cite{Castro-Cordoba-Fefferman-Gancedo-LopezFernandez:rayleigh-taylor-breakdown} then shows that for each $r>0$ there is $T_r$ (depending only on $r$) such that \eqref{muskatinterface} has a unique analytic solution $z_{N,A,r,a,c}$ on the time interval $(-T_r,T_r)$ with initial condition $z_{N,A,r,a,c}(x,0)$ (moreover, $\pa_t z_{N,A,r,a,c}$ is also analytic), and this satisfies for each $k\ge 0$, 
\[
\sup_{N\ge 1 \,\&\, A\in[1.08050, 1.08055] \,\&\, r,a,c\in[0,1/2] \,\&\, |t|<T_r } \left( \|\pa_x^k z_{N,A,r,a,c}(\cdot,t)\|_{L^\infty} + \|\pa_t\pa_x^k z_{N,A,r,a,c}(\cdot,t)\|_{L^\infty} \right)<\infty.
\]
(Below we always consider  $A\in[1.08050, 1.08055]$ and $r,a,c\in[0,\frac12]$.)  

This means that the bounds \eqref{2.1} and \eqref{derivadauniforme} extend to each $z_{N,A,r,a,c}$ and $(x,t)\in\mathbb R\times(-T_r,T_r)$ (where $T_0=0$),  with a uniform $C$.  We also have
\begin{equation} \label{2.5}
\pa_{t} \pa_{\alx} z^{1}_{N,A,r,a,c}(4N\pi,0) \ge 10^{-6} \qquad \text{and} \qquad  \pa_{t} \pa_{\alx} z^{1}_{N,A,r,a,c}(\pm 2N\pi,0) \le -10^{-6},
\end{equation}
as well as
\begin{equation} \label{2.6}
\pa_{t} \pa_{\alx} z^{1}_{N,1.08050,r,a,c}(0,0) \ge 10^{-6} \qquad \text{and} \qquad  \pa_{t} \pa_{\alx} z^{1}_{N,1.08055,r,a,c}(0,0) \le -10^{-6},
\end{equation}
both when $N^{-1}+r+a+c$ is small enough.  This follows from the bounds \eqref{2.4} and from
\begin{equation} \label{2.7}
\| \pa_x^k z_{N,A,r,a,c}(\cdot,0) - \pa_x^k \bar z_{N,A}\|_{L^\infty(I_N)}\to 0 \qquad\text{as $N^{-1}+r+a+c\to 0$}
\end{equation}
for each $k$, where $I_N=\bigcup_{n\in\mathbb Z} \left(2N\pi n-N,2N\pi n+N\right)$.  Similarly,  \eqref{2.2} and \eqref{2.7} also prove
\begin{equation} \label{2.8}
\pa_{tt} \pa_{\alx} z^{1}_{N,A,r,a,c}(0,0)\ge 20
\end{equation}
for small enough $N^{-1}+r+a+c$.  Fix now $N$ so that \eqref{2.5}, \eqref{2.6}, and \eqref{2.8} hold for all small enough $r+a+c$.

We next notice that for each $r>0$ we have $\pa_{\alx} z^{1}_{N,A,r,0,0}(x)=1-(P_r*\cos) (x)$, which is a  $2\pi$-periodic function with a positive minimum at $x=0$ (independent of $N,A$).  Thus there is a unique $a_{r}>0$ (small if $r>0$ is small, and also dependent on the fixed $N$) such that $\pa_{\alx} z^{1}_{N,A,r,a_{r},a_{r}} (\cdot,0)\ge 0$,
\[
\pa_{\alx} z^{1}_{N,A,r,a_{r},a_{r}}(2N\pi n,0)=0 
\]
for each $n\in\mathbb Z$, and  
\[
\pa_{\alx} z^{1}_{N,A,r,a_{r},a_{r}}(x,0)>0
\]
for $x\notin 2N\pi\mathbb Z$.  Finally, for any $\delta,\ep\in[0,a_r)$ let $A_{r,\delta,\ep}\in(1.08050, 1.08055)$ be such that
\begin{equation} \label{2.9}
\pa_{t} \pa_{\alx} z^{1}_{N,A_{r,\delta,\ep},r,a_{r}-\delta,a_{r}-\ep}(0,0)=0,
\end{equation}
which exists due to \eqref{2.6} and continuity of $\pa_{t} \pa_{\alx} z^{1}_{N,A,r,a_{r}-\delta,a_{r}-\ep}(0,0)$ in $A$.

For the sake of simplicity of notation, let us denote $z=z_{N,A_{r,\delta,\ep},r,a_{r}+\delta,a_{r}-\ep}$.
If now $r>0$ is small enough and $\delta,\ep>0$ are also small enough (the bound on them depends on $r$ and also on the constant $C$ from \eqref{2.1} and \eqref{derivadauniforme} for $z_{N,A,r,a,c}$), and such that
\[
0<-\pa_{\alx} z^{1}(0,0) \ll \left[ \frac 1C \min_{n\in\{-1,1,2\}}\pa_{\alx} z^{1}(2N\pi n,0) \right]^2
\]
(that is, $\ep>0$ is small enough and $\delta\approx \frac{3\ep}{8N}$; moreover, then all three values inside the $\min$ are $\approx \frac\ep\pi$), then $z$ is the desired solution.  Indeed, $\pa_x z^1(0,t)<0$ for all small enough $|t|$ and the argument from the proof of Theorem \ref{theoremshifting} shows that $\pa_x z^1(x,t)>0$ for all $x\in\mathbb R$ when
\[
 |\pa_{\alx} z^{1}(0,0)|^{1/2} \ll |t| \ll \frac\ep{C\pi}. 
\]
Finally, \eqref{2.5} and a uniform bound on $\pa_{t}^2 \pa_{\alx} z^{1}_{N,A,r,a,c}$ (obtained similarly to \eqref{2.2}) show that
\[
\pa_{\alx} z^{1}(4N\pi,-t) <0 \qquad \text{and} \qquad \pa_{\alx} z^{1}(\pm 2N\pi,t) <0
\]
for all $t\approx 10^6 \ep$ if $\ep\ll10^{-6}T_r$ is small enough (because $\pa_{\alx} z^{1}(4N\pi,0)\approx \frac\ep\pi\approx \pa_{\alx} z^{1}(\pm2N\pi,0$)).

We thus proved the following result.

\begin{theorem} \label{T.1.1}
There exist $T>T'>\gamma>0$  and a spatially analytic solution $z$ to \eqref{muskatinterfaceperiodic} on the time interval $[-T,T]$  such that $z(\cdot,t)$ is a graph of a smooth function of $x$ when $|t|\in[T'-\gamma,T'+\gamma]$
 but $z(\cdot,t)$ is not a graph of a function of $x$ when $|t|\in [0,\gamma]\cup[T-\gamma,T]$.
\end{theorem}

\section{Technical details of the numerical implementation}
\label{SectionTechnicalDetails}

In this section, we give some technical details of the implementation of the computer-assisted part of the proof of Lemma \ref{lemacomputer}. In order to perform the rigorous computations we used the C-XSC library \cite{CXSC}. We refer the reader to the appendices to see a detailed expression of the integral terms involved in the calculations. For the sake of readability, we kept the same names for the integrals in the paper and in the code. The code can be found in the supplementary material.

The implementation is split into several files, and many of the headers of the functions (such as the integration methods) contain pointers to functions (the integrands) so that they can be reused for an arbitrary number of integrals with minimal changes and easy and safe debugging. For the sake of clarity, and at the cost of numerical performance and duplicity in the code, we decided to treat many simple integrals instead of a single big one.

We start discussing the details of the first part of Lemma \ref{lemacomputer}, corresponding to the one dimensional integrals. The 3 integrals can be found in Appendix \ref{appendixa}. We split them into two parts: a nonsingular one over the interval $[\delta, \pi]$ and a singular one over the interval $[0,\delta]$. The nonsingular part is calculated using a Gauss-Legendre quadrature of order 2, given by
\begin{multline*}
\int_{a}^{b} f(\eta) d\eta \in  \frac{b-a}{2}\left(f\left(\frac{b-a}{2}\frac{\sqrt{3}}{3} + \frac{b+a}{2}\right)+f\left(-\frac{b-a}{2}\frac{\sqrt{3}}{3}+ \frac{b+a}{2}\right)\right)
+\frac{1}{4320}(b-a)^{5}f^{4}([a,b]).
\end{multline*}
Moreover, the integration is done in an adaptive way. For each region, we accepted or rejected the result depending on the width in an absolute and a relative way. It is important to notice that because of the uncertainty in $\ep$ and/or overestimation, division by zero might occur, even in small integration intervals. We used $\delta = 2^{-9}$ and tolerances \texttt{AbsTol} and \texttt{RelTol} equal to $10^{-6}$.

In the singular region, the singularity around $y = 0$ is integrable, hence the integral is finite. We performed a Taylor expansion around $y = 0$ in both the numerator and denominator (resp. of order 2,2 and 4 for $A_1$, $A_2$ and $A_3$), simplified the powers of $y$ and then integrated. Potentially this could fail because the uncertainty in the parameters or overestimation could yield a Taylor series in which $0$ belongs to the coefficient of the first non-simplified power of the denominator. Whenever this happens, we try to integrate instead using a Gauss-Legendre quadrature of order 2.

The maximum number of subdivision levels was 18 ($2^{18}$ intervals) for the bounded region and 12 ($2^{12}$ intervals) for the singular one. The splitting of the intervals is done in an arithmetic way, i.e, we split an integration interval $[a,b]$ into $\left[a,\frac{a+b}{2}\right]$ and $\left[\frac{a+b}{2},b\right]$.

In the second part of Lemma \ref{lemacomputer} we have to deal with 41 two dimensional integrals (see Appendix \ref{appendixb} for a detailed list of them and their derivation). The first step is to exploit the symmetry of the integrands in $(y,z)$ variables to integrate only over the region $[0,\pi] \times [-\pi,\pi]$. We will distinguish four different regions labeled in the following way: nonsingular $([\delta,\pi] \times [\delta,\pi]) \cup ([\delta,\pi] \times [-\pi,-\delta])$, singular-first-coordinate $([0,\delta] \times [\delta,\pi]) \cup ([0,\delta] \times [-\pi,-\delta]) $, singular-second-coordinate $[\delta,\pi] \times [-\delta,\delta] $ and singular-center $[0,\delta] \times [-\delta,\delta]$.

The nonsingular region was integrated as before, using a 2D Gauss-Legendre quadrature of order 2. The singular-center region was integrated in the following way. Assuming $\sign(a) = \sign(b), \sign(c) = \sign(d)$ and that we expand up to orders $num\_y, den\_y, num\_z, den\_z$:
\begin{align*}
 \int_{a}^{b} \int_{c}^{d} \frac{\text{Num}(y,z)}{\text{Den}(y,z)}dy dz
& \in \int_{a}^{b} \int_{c}^{d} \frac{\frac{1}{num\_y!num\_z!}\pa_{y}^{num\_y}\pa_{z}^{num\_z}\text{Num}(A,B)y^{num\_y}z^{num\_z}}{\frac{1}{den\_y!den\_z!}\pa_{y}^{den\_y}\pa_{z}^{den\_z}\text{Den}(A,B)y^{den\_y}z^{den\_z}}dy dz \\
& = \frac{1}{1+num\_y-den\_y}\frac{1}{1+num\_z-den\_z}\frac{den\_y!den\_z!}{num\_y!num\_z!} \\
& \left.\left.\times \frac{\pa_{y}^{num\_y}\pa_{z}^{num\_z}\text{Num}(A,B)}{\pa_{y}^{den\_y}\pa_{z}^{den\_z}\text{Den}(A,B)}y^{1+num\_y-den\_y}z^{1+num\_z-den\_z}\right|_{z=c}^{z=d}\right|_{y=a}^{y=b}
\end{align*}
where $A$ is the convex hull of $\{0, a, b\}$ and $B$ is the convex hull of $\{0, c, d\}$. For the singular-first-coordinate and singular-second-coordinate regions the same procedure was applied taking $num\_z = den\_z = 0, B = [c,d]$ and $num\_y = den\_y = 0, A = [a,b]$ respectively. A detailed list of the orders of each of the integrals can be found in the appendix in Table \ref{table_taylor}. Whenever the Taylor-based formulas failed because of 0 being enclosed in the denominator terms, we tried to integrate using 2D Gauss-Legendre of order 2.

In this two dimensional setting, we used a geometric splitting (in both coordinates) in the nonsingular region, arithmetic in the singular-center and singular-first-coordinate and hybrid in the singular-second-coordinate (see below). The geometric splitting consists in splitting by the geometric mean as opposed to the arithmetic one (i.e. assuming $a$ and $b$ have the same sign and are non-zero, we split $[a,b]$ into $[a,\sqrt{ab} \text{ sign}(a)]$ and $[\sqrt{ab}\text{ sign}(a),b]$). While the arithmetic division minimizes the length of the longest piece after the division, the geometric one minimizes the piece with the biggest ratio between its endpoints. This can be particularly useful in many cases: for example in order to avoid divisions by zero for integrands of the type $\frac{1}{y-A\sin(y)}$, which is a simplified version of some of the denominators that appear in all the terms. Detailed results of the breakdown by region and by term can be found in Table \ref{table_results}.

We chose $\delta = 2^{-5}$, and \texttt{AbsTol} and \texttt{RelTol} equal to $10^{-4}$. We changed the number of maximum subdivision levels depending on the region and (possibly) depending on the terms. For the nonsingular region, the maximum number was 10 ($2^{20}$ intervals). In the singular-first-coordinate, the maximum number of subdivisions was 8 ($2^{16}$ intervals), and that number was also used in the singular-center region. The singular-second-coordinate region was treated differently: all terms other than $B_{47}$ and $B_{55}$ were further split into 3 subregions: $[\delta,0.65] \times [-\delta,\delta]$, $[0.65,0.95] \times [-\delta,\delta]$ and $[0.95,\pi] \times [-\delta,\delta]$ and setting the maximum number of subdivisions to 9 in each subregion. The first and second subregion were computed using arithmetic splitting, whereas the third one was split geometrically only in the first coordinate, and arithmetically in the second.

 The singular-second-coordinate regions of the terms $B_{47}$ and $B_{55}$ are highly singular because of the cubic denominators and they required special precision. They were subdivided into 6 subregions: namely $[\delta,0.325] \times [-\delta,\delta]$, $[0.325,0.65] \times [-\delta,\delta] $, $[0.65,0.775] \times [-\delta,\delta]$, $[0.775,0.95] \times [-\delta,\delta]$, $[0.95,1.5] \times [-\delta,\delta] $ and $[1.5,\pi] \times [-\delta,\delta] $. The maximum number of subdivisions was 10 in each subregion. The last 2 subregions were split geometrically in the first coordinate, arithmetically in the second. The other 4 subregions were split arithmetically in each of the coordinates.

 The simulations were run on the NewComp cluster at Princeton University. Each of the programs was run on a core of 2 Xeon X5680 CPUs (6 cores each, 12 in total) at 3.33 GHz and 8 GB of RAM. The total runtime was about 3.5 min for the first part of Lemma \ref{lemacomputer}. For the second part, the different runtimes are summarized in table \ref{tableruntime}.

\section*{Acknowledgements}

DC and JGS were partially supported by the grant MTM2014-59488-P (Spain) and ICMAT Severo Ochoa project SEV-2011-0087. JGS was partially supported by an AMS-Simons Travel Grant. AZ acknowledges partial support by NSF grant DMS-1056327. We thank Princeton University for computing facilities (NewComp cluster).

\newpage
\appendix

\section{Integrals needed for the calculation of $\pa_{tx}z^{1}(0,0)$}

\label{appendixa}

We start with 
\begin{align*}
 \pa_{t}z^{1}(x,t) & = \int_{\T} \frac{\sin(z^1(x) - z^1(x-y))(z_{x}^{1}(x) - z_{x}^{1}(x-y))}{\cosh(z^2(x) - z^2(x-y)) - \cos(z^1(x) - z^1(x-y))}dy
\end{align*}

After taking a derivative in space:
\begin{align*}
 \pa_{tx}z^{1}(x,0) & = \int_{\T} \frac{\sin(z^1(x) - z^1(x-y))(z_{xx}^{1}(x) - z_{xx}^{1}(x-y))}{\cosh(z^2(x) - z^2(x-y)) - \cos(z^1(x) - z^1(x-y))}dy \\
& + \int_{\T} \frac{\cos(z^1(x) - z^1(x-y))(z_{x}^{1}(x) - z_{x}^{1}(x-y))^{2}}{\cosh(z^2(x) - z^2(x-y)) - \cos(z^1(x) - z^1(x-y))}dy \\
& - \int_{\T} \frac{\sin(z^1(x) - z^1(x-y))(z_{x}^{1}(x) - z_{x}^{1}(x-y))}{(\cosh(z^2(x) - z^2(x-y)) - \cos(z^1(x) - z^1(x-y)))^{2}} \\
& \times (\sinh(z^{2}(x) - z^{2}(x-y))(z_{x}^{2}(x) - z_{x}^{2}(x-y)) + \sin(z^{1}(x)-z^{1}(x-y))(z_{x}^{1}(x) - z_{x}^{1}(x-y)))dy \\
\end{align*}

Evaluating at $x = 0$ and exploiting the symmetry of the integral:
\begin{align*}
 \pa_{tx}z^{1}(0,0) & = \int_{\T} \frac{\sin(z^1(y))(z_{xx}^{1}(0) + z_{xx}^{1}(y))}{\cosh( z^2(y)) - \cos( z^1(y))}dy \\
& + \int_{\T} \frac{\cos(z^1(y))(z_{x}^{1}(0) - z_{x}^{1}(y))^{2}}{\cosh(z^2(y)) - \cos(z^1(y))}dy \\
& - \int_{\T} \frac{\sin(z^1(y))(z_{x}^{1}(0) - z_{x}^{1}(y))}{(\cosh(z^2(y)) - \cos(z^1(y)))^{2}}(\sinh(z^{2}(y))(z_{x}^{2}(0) - z_{x}^{2}(y)) + \sin(z^{1}(y))(z_{x}^{1}(0) - z_{x}^{1}(y)))dy \\
& = A_{1} + A_{2} + A_{3}
\end{align*}
\begin{align*}
 A_{1} & = 2\int_{0}^{\pi} \frac{\sin(z^1(y))(z_{xx}^{1}(y))}{\cosh( z^2(y)) - \cos( z^1(y))}dy \\
A_{2} & = 2\int_{0}^{\pi} \frac{\cos(z^1(y))(z_{x}^{1}(0) - z_{x}^{1}(y))^{2}}{\cosh(z^2(y)) - \cos(z^1(y))}dy \\
A_{3} & = -2 \int_{0}^{\pi} \frac{\sin(z^1(y))(z_{x}^{1}(0) - z_{x}^{1}(y))}{(\cosh(z^2(y)) - \cos(z^1(y)))^{2}}(\sinh(z^{2}(y))(z_{x}^{2}(0) - z_{x}^{2}(y)) + \sin(z^{1}(y))(z_{x}^{1}(0) - z_{x}^{1}(y)))dy
\end{align*}

\section{Integrals needed for the calculation of $\pa_{ttx}z^{1}(0,0)$}

\label{appendixb}

After taking a derivative in time:
\begin{align*}
\pa_{tt}z^{1}(x,t) & = \int_{\T} \frac{\cos(z^1(x) - z^1(x-y))(z_{t}^1(x) - z_{t}^1(x-y))(z_{x}^{1}(x) - z_{x}^{1}(x-y))}{\cosh(z^2(x) - z^2(x-y)) - \cos(z^1(x) - z^1(x-y))}dy \\
& + \int_{\T} \frac{\sin(z^1(x) - z^1(x-y))(z^{1}_{tx}(x) - z^{1}_{tx}(x-y))}{\cosh(z^2(x) - z^2(x-y)) - \cos(z^1(x) - z^1(x-y))}dy \\
& - \int_{\T} \frac{\sin(z^1(x) - z^1(x-y))(z^{1}_{x}(x) - z^{1}_{x}(x-y))\sinh(z^{2}(x) - z^{2}(x-y))(z_{t}^{2}(x) - z_{t}^{2}(x-y)}{(\cosh(z^2(x) - z^2(x-y)) - \cos(z^1(x) - z^1(x-y)))^{2}}dy \\
& - \int_{\T} \frac{\sin(z^1(x) - z^1(x-y))(z^{1}_{x}(x) - z^{1}_{x}(x-y))\sin(z^{1}(x)-z^{1}(x-y))(z_{t}^{1}(x) - z_{t}^{1}(x-y))}{(\cosh(z^2(x) - z^2(x-y)) - \cos(z^1(x) - z^1(x-y)))^{2}}dy \\
& = I_{1}(x) + I_{2}(x) + I_{3}(x) + I_{4}(x)
\end{align*}

We can further develop the terms of the second derivative:
\begin{multline*}
 I_{1}(x) = \int_{\T} \int_{\T} \frac{\cos(z^1(x) - z^1(x-y)))(z^{1}_{x}(x) - z^{1}_{x}(x-y))}{\cosh(z^2(x) - z^2(x-y)) - \cos(z^1(x) - z^1(x-y))} \\
 \times \left(\frac{\sin(z^1(x) - z^1(x-z))(z^{1}_{x}(x) - z^{1}_{x}(x-z))}{\cosh(z^2(x) - z^2(x-z)) - \cos(z^1(x) - z^1(x-z))}\right.
\\
\left. - 
\frac{\sin(z^1(x-y) - z^1(x-y-z))(z^{1}_{x}(x-y) - z^{1}_{x}(x-y-z))}{\cosh(z^2(x-y) - z^2(x-y-z)) - \cos(z^1(x-y) - z^1(x-y-z))}
\right)
dy dz 
\end{multline*}

\begin{align*}
 I_{2}(x) = I_{21}(x) + I_{22}(x) + I_{23}(x) + I_{24}(x),
\end{align*}
where
\begin{multline*}
 I_{21}(x) = \int_{\T} \int_{\T} \frac{\sin(z^1(x) - z^1(x-y)))}{\cosh(z^2(x) - z^2(x-y)) - \cos(z^1(x) - z^1(x-y))}\\
 \times \left(\frac{\cos(z^1(x) - z^1(x-z))(z^{1}_{x}(x) - z^{1}_{x}(x-z))(z_{x}^1(x) - z_{x}^1(x-z))}{\cosh(z^2(x) - z^2(x-z)) - \cos(z^1(x) - z^1(x-z))}\right.
\\
\left. - 
\frac{\cos(z^1(x-y) - z^1(x-y-z))(z^{1}_{x}(x-y) - z^{1}_{x}(x-y-z))(z_{x}^1(x-y) - z_{x}^1(x-y-z))}{\cosh(z^2(x-y) - z^2(x-y-z)) - \cos(z^1(x-y) - z^1(x-y-z))}
\right)
dy dz 
\end{multline*}

\begin{multline*}
 I_{22}(x) = \int_{\T} \int_{\T} \frac{\sin(z^1(x) - z^1(x-y)))}{\cosh(z^2(x) - z^2(x-y)) - \cos(z^1(x) - z^1(x-y))}\\
 \times \left(\frac{\sin(z^1(x) - z^1(x-z))(z^{1}_{xx}(x) - z^{1}_{xx}(x-z))}{\cosh(z^2(x) - z^2(x-z)) - \cos(z^1(x) - z^1(x-z))}\right.
\\
\left. - 
\frac{\sin(z^1(x-y) - z^1(x-y-z))(z^{1}_{xx}(x-y) - z^{1}_{xx}(x-y-z))}{\cosh(z^2(x-y) - z^2(x-y-z)) - \cos(z^1(x-y) - z^1(x-y-z))}
\right)
dy dz 
\end{multline*}

\begin{multline*}
 I_{23}(x) = -\int_{\T} \int_{\T} \frac{\sin(z^1(x) - z^1(x-y)))}{\cosh(z^2(x) - z^2(x-y)) - \cos(z^1(x) - z^1(x-y))}\\ 
\times \left(\frac{\sin(z^1(x) - z^1(x-z))(z^{1}_{x}(x) - z^{1}_{x}(x-z))\sinh(z^2(x) - z^2(x-z))(z^{2}_{x}(x) - z^{2}_{x}(x-z))}{(\cosh(z^2(x) - z^2(x-z)) - \cos(z^1(x) - z^1(x-z)))^{2}}\right.
\\
\left. - 
\frac{\sin(z^1(x-y) - z^1(x-y-z))(z^{1}_{x}(x-y) - z^{1}_{x}(x-y-z))\sinh(z^2(x-y) - z^2(x-y-z))(z^{2}_{x}(x-y) - z^{2}_{x}(x-y-z))}{(\cosh(z^2(x-y) - z^2(x-y-z)) - \cos(z^1(x-y) - z^1(x-y-z)))^2}
\right)
dy dz 
\end{multline*}

\begin{multline*}
 I_{24}(x) = -\int_{\T} \int_{\T} \frac{\sin(z^1(x) - z^1(x-y)))}{\cosh(z^2(x) - z^2(x-y)) - \cos(z^1(x) - z^1(x-y))}\\ 
\times \left(\frac{\sin(z^1(x) - z^1(x-z))(z^{1}_{x}(x) - z^{1}_{x}(x-z))\sin(z^1(x) - z^1(x-z))(z^{1}_{x}(x) - z^{1}_{x}(x-z))}{(\cosh(z^2(x) - z^2(x-z)) - \cos(z^1(x) - z^1(x-z)))^{2}}\right.
\\
\left. - 
\frac{\sin(z^1(x-y) - z^1(x-y-z))(z^{1}_{x}(x-y) - z^{1}_{x}(x-y-z))\sin(z^1(x-y) - z^1(x-y-z))(z^{1}_{x}(x-y) - z^{1}_{x}(x-y-z))}{(\cosh(z^2(x-y) - z^2(x-y-z)) - \cos(z^1(x-y) - z^1(x-y-z)))^2}
\right)
dy dz 
\end{multline*}

\begin{multline*}
 I_{3}(x)  = - \int_{\T} \int_{\T} \frac{\sin(z^1(x) - z^1(x-y))(z^{1}_{x}(x) - z^{1}_{x}(x-y))\sinh(z^{2}(x) - z^{2}(x-y))}{(\cosh(z^2(x) - z^2(x-y)) - \cos(z^1(x) - z^1(x-y)))^{2}} \\
 \times \left(\frac{\sin(z^1(x) - z^1(x-z))(z^{2}_{x}(x) - z^{2}_{x}(x-z))}{\cosh(z^2(x) - z^2(x-z)) - \cos(z^1(x) - z^1(x-z))}\right.
\\
\left. - 
\frac{\sin(z^1(x-y) - z^1(x-y-z))(z^{2}_{x}(x-y) - z^{2}_{x}(x-y-z))}{\cosh(z^2(x-y) - z^2(x-y-z)) - \cos(z^1(x-y) - z^1(x-y-z))}
\right)
dy dz 
\end{multline*}

\begin{multline*}
 I_{4}(x)  = - \int_{\T} \int_{\T} \frac{\sin(z^1(x) - z^1(x-y))(z^{1}_{x}(x) - z^{1}_{x}(x-y))\sin(z^{1}(x)-z^{1}(x-y))}{(\cosh(z^2(x) - z^2(x-y)) - \cos(z^1(x) - z^1(x-y)))^{2}} \\
 \times \left(\frac{\sin(z^1(x) - z^1(x-z))(z^{1}_{x}(x) - z^{1}_{x}(x-z))}{\cosh(z^2(x) - z^2(x-z)) - \cos(z^1(x) - z^1(x-z))}\right.
\\
\left. - 
\frac{\sin(z^1(x-y) - z^1(x-y-z))(z^{1}_{x}(x-y) - z^{1}_{x}(x-y-z))}{\cosh(z^2(x-y) - z^2(x-y-z)) - \cos(z^1(x-y) - z^1(x-y-z))}
\right)
dy dz 
\end{multline*}

We now compute $\pa_{x}$ of the integrals:
\begin{align*}
 \left.\pa_{x} I_{1}(x)\right|_{x=0} = \left.B_{11}(x) + B_{12}(x) + B_{13}(x) + B_{14}(x) + B_{15}(x) + B_{16}(x)\right|_{x=0}
\end{align*}
We have:
\begin{multline*}
 B_{11}(x) = -\int_{\T} \int_{\T} \frac{\sin(z^1(x) - z^1(x-y)))(z^{1}_{x}(x) - z^{1}_{x}(x-y))^{2}}{\cosh(z^2(x) - z^2(x-y)) - \cos(z^1(x) - z^1(x-y))} \\
 \times \left(\frac{\sin(z^1(x) - z^1(x-z))(z^{1}_{x}(x) - z^{1}_{x}(x-z))}{\cosh(z^2(x) - z^2(x-z)) - \cos(z^1(x) - z^1(x-z))}\right.
\\
\left. - 
\frac{\sin(z^1(x-y) - z^1(x-y-z))(z^{1}_{x}(x-y) - z^{1}_{x}(x-y-z))}{\cosh(z^2(x-y) - z^2(x-y-z)) - \cos(z^1(x-y) - z^1(x-y-z))}
\right)
dy dz 
\end{multline*}

\begin{multline*}
 B_{12}(x)  = \int_{\T} \int_{\T} \frac{\cos(z^1(x) - z^1(x-y)))(z^{1}_{xx}(x) - z^{1}_{xx}(x-y))}{\cosh(z^2(x) - z^2(x-y)) - \cos(z^1(x) - z^1(x-y))} \\
 \times \left(\frac{\sin(z^1(x) - z^1(x-z))(z^{1}_{x}(x) - z^{1}_{x}(x-z))}{\cosh(z^2(x) - z^2(x-z)) - \cos(z^1(x) - z^1(x-z))}\right.
\\
\left. - 
\frac{\sin(z^1(x-y) - z^1(x-y-z))(z^{1}_{x}(x-y) - z^{1}_{x}(x-y-z))}{\cosh(z^2(x-y) - z^2(x-y-z)) - \cos(z^1(x-y) - z^1(x-y-z))}
\right)
dy dz 
\end{multline*}

\begin{multline*}
 B_{13}(x) = -\int_{\T} \int_{\T} \frac{\cos(z^1(x) - z^1(x-y)))(z^{1}_{x}(x) - z^{1}_{x}(x-y))}{(\cosh(z^2(x) - z^2(x-y)) - \cos(z^1(x) - z^1(x-y)))^{2}} \\
\times \left(\sinh(z^{2}(x) - z^{2}(x-y))(z^{2}_{x}(x) - z^{2}_{x}(x-y)) + \sin(z^1(x) - z^1(x-y))(z^1_{x}(x) - z^1_{x}(x-y))\right) \\
 \times \left(\frac{\sin(z^1(x) - z^1(x-z))(z^{1}_{x}(x) - z^{1}_{x}(x-z))}{\cosh(z^2(x) - z^2(x-z)) - \cos(z^1(x) - z^1(x-z))}\right.
\\
\left. - 
\frac{\sin(z^1(x-y) - z^1(x-y-z))(z^{1}_{x}(x-y) - z^{1}_{x}(x-y-z))}{\cosh(z^2(x-y) - z^2(x-y-z)) - \cos(z^1(x-y) - z^1(x-y-z))}
\right)
dy dz 
\end{multline*}

\begin{multline*}
 B_{14}(x) = \int_{\T} \int_{\T} \frac{\cos(z^1(x) - z^1(x-y)))(z^{1}_{x}(x) - z^{1}_{x}(x-y))}{\cosh(z^2(x) - z^2(x-y)) - \cos(z^1(x) - z^1(x-y))} \\
 \times \left(\frac{\cos(z^1(x) - z^1(x-z))(z^{1}_{x}(x) - z^{1}_{x}(x-z))^{2}}{\cosh(z^2(x) - z^2(x-z)) - \cos(z^1(x) - z^1(x-z))}\right.
\\
\left. - 
\frac{\cos(z^1(x-y) - z^1(x-y-z))(z^{1}_{x}(x-y) - z^{1}_{x}(x-y-z))^{2}}{\cosh(z^2(x-y) - z^2(x-y-z)) - \cos(z^1(x-y) - z^1(x-y-z))}
\right)
dy dz 
\end{multline*}

\begin{multline*}
 B_{15}(x) = \int_{\T} \int_{\T} \frac{\cos(z^1(x) - z^1(x-y)))(z^{1}_{x}(x) - z^{1}_{x}(x-y))}{\cosh(z^2(x) - z^2(x-y)) - \cos(z^1(x) - z^1(x-y))} \\
 \times \left(\frac{\sin(z^1(x) - z^1(x-z))(z^{1}_{xx}(x) - z^{1}_{xx}(x-z))}{\cosh(z^2(x) - z^2(x-z)) - \cos(z^1(x) - z^1(x-z))}\right.
\\
\left. - 
\frac{\sin(z^1(x-y) - z^1(x-y-z))(z^{1}_{xx}(x-y) - z^{1}_{xx}(x-y-z))}{\cosh(z^2(x-y) - z^2(x-y-z)) - \cos(z^1(x-y) - z^1(x-y-z))}
\right)
dy dz 
\end{multline*}

\begin{multline*}
 B_{16}(x) = -\int_{\T} \int_{\T} \frac{\cos(z^1(x) - z^1(x-y)))(z^{1}_{x}(x) - z^{1}_{x}(x-y))}{\cosh(z^2(x) - z^2(x-y)) - \cos(z^1(x) - z^1(x-y))} \\
 \times \left(\frac{\sin(z^1(x) - z^1(x-z))(z^{1}_{x}(x) - z^{1}_{x}(x-z))}{(\cosh(z^2(x) - z^2(x-z)) - \cos(z^1(x) - z^1(x-z)))^{2}} \right.\\
\times(\sinh(z^{2}(x) - z^{2}(x-z))(z^{2}_{x}(x)-z^{2}_{x}(x-z)) + \sin(z^{1}(x) - z^{1}(x-z))(z^{1}_{x}(x)-z_{x}^{1}(x-z)))
\\ - 
\frac{\sin(z^1(x-y) - z^1(x-y-z))(z^{1}_{x}(x-y) - z^{1}_{x}(x-y-z))}{(\cosh(z^2(x-y) - z^2(x-y-z)) - \cos(z^1(x-y) - z^1(x-y-z)))^{2}} \\
\left.\times
(\sinh(z^{2}(x-y) - z^{2}(x-y-z))(z^{2}_{x}(x-y)-z^{2}_{x}(x-y-z)) + \sin(z^{1}(x-y) - z^{1}(x-y-z))(z^{1}_{x}(x-y)-z_{x}^{1}(x-y-z)))\right)
dy dz 
\end{multline*}

We move on to $I_{21}(x)$. Taking a derivative yields:
\begin{align*}
 \pa_{x} I_{21}(x)  = B_{21}(x) + B_{22}(x) + B_{23}(x) + B_{24}(x) + B_{25}(x),
\end{align*}
where
\begin{multline*}
 B_{21}(x) = \int_{\T} \int_{\T} \frac{\cos(z^1(x) - z^1(x-y))(z^1_{x}(x) - z^1_{x}(x-y))}{\cosh(z^2(x) - z^2(x-y)) - \cos(z^1(x) - z^1(x-y))}\\
 \times \left(\frac{\cos(z^1(x) - z^1(x-z))(z^{1}_{x}(x) - z^{1}_{x}(x-z))^{2}}{\cosh(z^2(x) - z^2(x-z)) - \cos(z^1(x) - z^1(x-z))}\right.
\\
\left. - 
\frac{\cos(z^1(x-y) - z^1(x-y-z))(z^{1}_{x}(x-y) - z^{1}_{x}(x-y-z))^{2}}{\cosh(z^2(x-y) - z^2(x-y-z)) - \cos(z^1(x-y) - z^1(x-y-z))}
\right)
dy dz 
\end{multline*}

\begin{multline*}
 B_{22}(x) = -\int_{\T} \int_{\T} \frac{\sin(z^1(x) - z^1(x-y))}{(\cosh(z^2(x) - z^2(x-y)) - \cos(z^1(x) - z^1(x-y)))^{2}}\\
\times \left(\sinh(z^{2}(x) - z^{2}(x-y))(z^{2}_{x}(x) - z^{2}_{x}(x-y)) + \sin(z^1(x) - z^1(x-y))(z^1_{x}(x) - z^1_{x}(x-y))\right) \\
 \times \left(\frac{\cos(z^1(x) - z^1(x-z))(z^{1}_{x}(x) - z^{1}_{x}(x-z))^{2}}{\cosh(z^2(x) - z^2(x-z)) - \cos(z^1(x) - z^1(x-z))}\right.
\\
\left. - 
\frac{\cos(z^1(x-y) - z^1(x-y-z))(z^{1}_{x}(x-y) - z^{1}_{x}(x-y-z))^{2}}{\cosh(z^2(x-y) - z^2(x-y-z)) - \cos(z^1(x-y) - z^1(x-y-z))}
\right)
dy dz 
\end{multline*}

\begin{multline*}
 B_{23}(x) = -\int_{\T} \int_{\T} \frac{\sin(z^1(x) - z^1(x-y)))}{\cosh(z^2(x) - z^2(x-y)) - \cos(z^1(x) - z^1(x-y))}\\
 \times \left(\frac{\sin(z^1(x) - z^1(x-z))(z^{1}_{x}(x) - z^{1}_{x}(x-z))^{3}}{\cosh(z^2(x) - z^2(x-z)) - \cos(z^1(x) - z^1(x-z))}\right.
\\
\left. - 
\frac{\sin(z^1(x-y) - z^1(x-y-z))(z^{1}_{x}(x-y) - z^{1}_{x}(x-y-z))^{3}}{\cosh(z^2(x-y) - z^2(x-y-z)) - \cos(z^1(x-y) - z^1(x-y-z))}
\right)
dy dz 
\end{multline*}

\begin{multline*}
 B_{24}(x) = 2\int_{\T} \int_{\T} \frac{\sin(z^1(x) - z^1(x-y)))}{\cosh(z^2(x) - z^2(x-y)) - \cos(z^1(x) - z^1(x-y))}\\
 \times \left(\frac{\cos(z^1(x) - z^1(x-z))(z^{1}_{xx}(x) - z^{1}_{xx}(x-z))(z_{x}^1(x) - z_{x}^1(x-z))}{\cosh(z^2(x) - z^2(x-z)) - \cos(z^1(x) - z^1(x-z))}\right.
\\
\left. - 
\frac{\cos(z^1(x-y) - z^1(x-y-z))(z^{1}_{xx}(x-y) - z^{1}_{xx}(x-y-z))(z_{x}^1(x-y) - z_{x}^1(x-y-z))}{\cosh(z^2(x-y) - z^2(x-y-z)) - \cos(z^1(x-y) - z^1(x-y-z))}
\right)
dy dz 
\end{multline*}

\begin{multline*}
 B_{25}(x) = -\int_{\T} \int_{\T} \frac{\sin(z^1(x) - z^1(x-y)))}{\cosh(z^2(x) - z^2(x-y)) - \cos(z^1(x) - z^1(x-y))}\\
 \times \left(\frac{\cos(z^1(x) - z^1(x-z))(z^{1}_{x}(x) - z^{1}_{x}(x-z))^{2}}{(\cosh(z^2(x) - z^2(x-z)) - \cos(z^1(x) - z^1(x-z)))^{2}}\right.
\\
\times(\sinh(z^{2}(x) - z^{2}(x-z))(z^{2}_{x}(x)-z^{2}_{x}(x-z)) + \sin(z^{1}(x) - z^{1}(x-z))(z^{1}_{x}(x)-z_{x}^{1}(x-z)))
\\ 
 - 
\frac{\cos(z^1(x-y) - z^1(x-y-z))(z^{1}_{x}(x-y) - z^{1}_{x}(x-y-z))^{2}}{(\cosh(z^2(x-y) - z^2(x-y-z)) - \cos(z^1(x-y) - z^1(x-y-z)))^{2}} \\
\left.\times
(\sinh(z^{2}(x-y) - z^{2}(x-y-z))(z^{2}_{x}(x-y)-z^{2}_{x}(x-y-z)) + \sin(z^{1}(x-y) - z^{1}(x-y-z))(z^{1}_{x}(x-y)-z_{x}^{1}(x-y-z)))\right)
dy dz 
\end{multline*}

Next we differentiate $I_{22}(x)$:
\begin{align*}
 \pa_{x} I_{22}(x) = B_{31}(x) + B_{32}(x) + B_{33}(x) + B_{34}(x) + B_{35}(x),
\end{align*}
where
\begin{multline*}
 B_{31}(x) = \int_{\T} \int_{\T} \frac{\cos(z^1(x) - z^1(x-y))(z^1_{x}(x) - z^1_{x}(x-y))}{\cosh(z^2(x) - z^2(x-y)) - \cos(z^1(x) - z^1(x-y))}\\
 \times \left(\frac{\sin(z^1(x) - z^1(x-z))(z^{1}_{xx}(x) - z^{1}_{xx}(x-z))}{\cosh(z^2(x) - z^2(x-z)) - \cos(z^1(x) - z^1(x-z))}\right.
\\
\left. - 
\frac{\sin(z^1(x-y) - z^1(x-y-z))(z^{1}_{xx}(x-y) - z^{1}_{xx}(x-y-z))}{\cosh(z^2(x-y) - z^2(x-y-z)) - \cos(z^1(x-y) - z^1(x-y-z))}
\right)
dy dz 
\end{multline*}

\begin{multline*}
 B_{32}(x) = -\int_{\T} \int_{\T} \frac{\sin(z^1(x) - z^1(x-y)))}{(\cosh(z^2(x) - z^2(x-y)) - \cos(z^1(x) - z^1(x-y)))^{2}}\\
\times \left(\sinh(z^{2}(x) - z^{2}(x-y))(z^{2}_{x}(x) - z^{2}_{x}(x-y)) + \sin(z^1(x) - z^1(x-y))(z^{1}_{x}(x) - z^{1}_{x}(x-y))\right) \\
 \times \left(\frac{\sin(z^1(x) - z^1(x-z))(z^{1}_{xx}(x) - z^{1}_{xx}(x-z))}{\cosh(z^2(x) - z^2(x-z)) - \cos(z^1(x) - z^1(x-z))}\right.
\\
\left. - 
\frac{\sin(z^1(x-y) - z^1(x-y-z))(z^{1}_{xx}(x-y) - z^{1}_{xx}(x-y-z))}{\cosh(z^2(x-y) - z^2(x-y-z)) - \cos(z^1(x-y) - z^1(x-y-z))}
\right)
dy dz 
\end{multline*}

\begin{multline*}
 B_{33}(x) = \int_{\T} \int_{\T} \frac{\sin(z^1(x) - z^1(x-y)))}{\cosh(z^2(x) - z^2(x-y)) - \cos(z^1(x) - z^1(x-y))}\\
 \times \left(\frac{\cos(z^1(x) - z^1(x-z))(z^{1}_{xx}(x) - z^{1}_{xx}(x-z))(z^1_{x}(x) - z^1_{x}(x-z))}{\cosh(z^2(x) - z^2(x-z)) - \cos(z^1(x) - z^1(x-z))}\right.
\\
\left. - 
\frac{\cos(z^1(x-y) - z^1(x-y-z))(z^{1}_{xx}(x-y) - z^{1}_{xx}(x-y-z))(z^1_{x}(x-y) - z^1_{x}(x-y-z))}{\cosh(z^2(x-y) - z^2(x-y-z)) - \cos(z^1(x-y) - z^1(x-y-z))}
\right)
dy dz 
\end{multline*}

\begin{multline*}
 B_{34}(x) = \int_{\T} \int_{\T} \frac{\sin(z^1(x) - z^1(x-y)))}{\cosh(z^2(x) - z^2(x-y)) - \cos(z^1(x) - z^1(x-y))}\\
 \times \left(\frac{\sin(z^1(x) - z^1(x-z))(z^{1}_{xxx}(x) - z^{1}_{xxx}(x-z))}{\cosh(z^2(x) - z^2(x-z)) - \cos(z^1(x) - z^1(x-z))}\right.
\\
\left. - 
\frac{\sin(z^1(x-y) - z^1(x-y-z))(z^{1}_{xxx}(x-y) - z^{1}_{xxx}(x-y-z))}{\cosh(z^2(x-y) - z^2(x-y-z)) - \cos(z^1(x-y) - z^1(x-y-z))}
\right)
dy dz 
\end{multline*}

\begin{multline*}
 B_{35}(x) = -\int_{\T} \int_{\T} \frac{\sin(z^1(x) - z^1(x-y)))}{\cosh(z^2(x) - z^2(x-y)) - \cos(z^1(x) - z^1(x-y))}\\
 \times \left(\frac{\sin(z^1(x) - z^1(x-z))(z^{1}_{xx}(x) - z^{1}_{xx}(x-z))}{(\cosh(z^2(x) - z^2(x-z)) - \cos(z^1(x) - z^1(x-z)))^{2}}\right.
\\
\times(\sinh(z^{2}(x) - z^{2}(x-z))(z^{2}_{x}(x)-z^{2}_{x}(x-z)) + \sin(z^{1}(x) - z^{1}(x-z))(z^{1}_{x}(x)-z_{x}^{1}(x-z)))
\\  - 
\frac{\sin(z^1(x-y) - z^1(x-y-z))(z^{1}_{xx}(x-y) - z^{1}_{xx}(x-y-z))}{(\cosh(z^2(x-y) - z^2(x-y-z)) - \cos(z^1(x-y) - z^1(x-y-z)))^{2}} \\
\left.\times
(\sinh(z^{2}(x-y) - z^{2}(x-y-z))(z^{2}_{x}(x-y)-z^{2}_{x}(x-y-z)) + \sin(z^{1}(x-y) - z^{1}(x-y-z))(z^{1}_{x}(x-y)-z_{x}^{1}(x-y-z)))\right)
dy dz 
\end{multline*}

The differentiation of $I_{23}(x)$ follows:
\begin{align*}
\pa_{x} I_{23}(x) = B_{41}(x) + B_{42}(x) + B_{43}(x) + B_{44}(x) + B_{45}(x) + B_{46}(x) + B_{47}(x),
\end{align*}
where
\begin{multline*}
 B_{41}(x) = -\int_{\T} \int_{\T} \frac{\cos(z^1(x) - z^1(x-y))(z^1_{x}(x) - z^1_{x}(x-y))}{\cosh(z^2(x) - z^2(x-y)) - \cos(z^1(x) - z^1(x-y))}\\ 
\times \left(\frac{\sin(z^1(x) - z^1(x-z))(z^{1}_{x}(x) - z^{1}_{x}(x-z))\sinh(z^2(x) - z^2(x-z))(z^{2}_{x}(x) - z^{2}_{x}(x-z))}{(\cosh(z^2(x) - z^2(x-z)) - \cos(z^1(x) - z^1(x-z)))^{2}}\right.
\\
\left. - 
\frac{\sin(z^1(x-y) - z^1(x-y-z))(z^{1}_{x}(x-y) - z^{1}_{x}(x-y-z))\sinh(z^2(x-y) - z^2(x-y-z))(z^{2}_{x}(x-y) - z^{2}_{x}(x-y-z))}{(\cosh(z^2(x-y) - z^2(x-y-z)) - \cos(z^1(x-y) - z^1(x-y-z)))^2}
\right)
dy dz 
\end{multline*}

\begin{multline*}
 B_{42}(x) = \int_{\T} \int_{\T} \frac{\sin(z^1(x) - z^1(x-y)))}{(\cosh(z^2(x) - z^2(x-y)) - \cos(z^1(x) - z^1(x-y)))^{2}}\\ 
\times \left(\sinh(z^{2}(x) - z^{2}(x-y))(z^{2}_{x}(x) - z^{2}_{x}(x-y)) + \sin(z^1(x) - z^1(x-y))(z^{1}_{x}(x) - z^{1}_{x}(x-y))\right) \\
\times \left(\frac{\sin(z^1(x) - z^1(x-z))(z^{1}_{x}(x) - z^{1}_{x}(x-z))\sinh(z^2(x) - z^2(x-z))(z^{2}_{x}(x) - z^{2}_{x}(x-z))}{(\cosh(z^2(x) - z^2(x-z)) - \cos(z^1(x) - z^1(x-z)))^{2}}\right.
\\
\left. - 
\frac{\sin(z^1(x-y) - z^1(x-y-z))(z^{1}_{x}(x-y) - z^{1}_{x}(x-y-z))\sinh(z^2(x-y) - z^2(x-y-z))(z^{2}_{x}(x-y) - z^{2}_{x}(x-y-z))}{(\cosh(z^2(x-y) - z^2(x-y-z)) - \cos(z^1(x-y) - z^1(x-y-z)))^2}
\right)
dy dz 
\end{multline*}

\begin{multline*}
 B_{43}(x) = -\int_{\T} \int_{\T} \frac{\sin(z^1(x) - z^1(x-y)))}{\cosh(z^2(x) - z^2(x-y)) - \cos(z^1(x) - z^1(x-y))}\\ 
\times \left(\frac{\cos(z^1(x) - z^1(x-z))(z^{1}_{x}(x) - z^{1}_{x}(x-z))^{2}\sinh(z^2(x) - z^2(x-z))(z^{2}_{x}(x) - z^{2}_{x}(x-z))}{(\cosh(z^2(x) - z^2(x-z)) - \cos(z^1(x) - z^1(x-z)))^{2}}\right.
\\
\left. - 
\frac{\cos(z^1(x-y) - z^1(x-y-z))(z^{1}_{x}(x-y) - z^{1}_{x}(x-y-z))^{2}\sinh(z^2(x-y) - z^2(x-y-z))(z^{2}_{x}(x-y) - z^{2}_{x}(x-y-z))}{(\cosh(z^2(x-y) - z^2(x-y-z)) - \cos(z^1(x-y) - z^1(x-y-z)))^2}
\right)
dy dz 
\end{multline*}

\begin{multline*}
 B_{44}(x) = -\int_{\T} \int_{\T} \frac{\sin(z^1(x) - z^1(x-y)))}{\cosh(z^2(x) - z^2(x-y)) - \cos(z^1(x) - z^1(x-y))}\\ 
\times \left(\frac{\sin(z^1(x) - z^1(x-z))(z^{1}_{xx}(x) - z^{1}_{xx}(x-z))\sinh(z^2(x) - z^2(x-z))(z^{2}_{x}(x) - z^{2}_{x}(x-z))}{(\cosh(z^2(x) - z^2(x-z)) - \cos(z^1(x) - z^1(x-z)))^{2}}\right.
\\
\left. - 
\frac{\sin(z^1(x-y) - z^1(x-y-z))(z^{1}_{xx}(x-y) - z^{1}_{xx}(x-y-z))\sinh(z^2(x-y) - z^2(x-y-z))(z^{2}_{x}(x-y) - z^{2}_{x}(x-y-z))}{(\cosh(z^2(x-y) - z^2(x-y-z)) - \cos(z^1(x-y) - z^1(x-y-z)))^2}
\right)
dy dz 
\end{multline*}

\begin{multline*}
 B_{45}(x) = -\int_{\T} \int_{\T} \frac{\sin(z^1(x) - z^1(x-y)))}{\cosh(z^2(x) - z^2(x-y)) - \cos(z^1(x) - z^1(x-y))}\\ 
\times \left(\frac{\sin(z^1(x) - z^1(x-z))(z^{1}_{x}(x) - z^{1}_{x}(x-z))\cosh(z^2(x) - z^2(x-z))(z^{2}_{x}(x) - z^{2}_{x}(x-z))^{2}}{(\cosh(z^2(x) - z^2(x-z)) - \cos(z^1(x) - z^1(x-z)))^{2}}\right.
\\
\left. - 
\frac{\sin(z^1(x-y) - z^1(x-y-z))(z^{1}_{x}(x-y) - z^{1}_{x}(x-y-z))\cosh(z^2(x-y) - z^2(x-y-z))(z^{2}_{x}(x-y) - z^{2}_{x}(x-y-z))^{2}}{(\cosh(z^2(x-y) - z^2(x-y-z)) - \cos(z^1(x-y) - z^1(x-y-z)))^2}
\right)
dy dz 
\end{multline*}

\begin{multline*}
 B_{46}(x) = -\int_{\T} \int_{\T} \frac{\sin(z^1(x) - z^1(x-y)))}{\cosh(z^2(x) - z^2(x-y)) - \cos(z^1(x) - z^1(x-y))}\\ 
\times \left(\frac{\sin(z^1(x) - z^1(x-z))(z^{1}_{x}(x) - z^{1}_{x}(x-z))\sinh(z^2(x) - z^2(x-z))(z^{2}_{xx}(x) - z^{2}_{xx}(x-z))}{(\cosh(z^2(x) - z^2(x-z)) - \cos(z^1(x) - z^1(x-z)))^{2}}\right.
\\
\left. - 
\frac{\sin(z^1(x-y) - z^1(x-y-z))(z^{1}_{x}(x-y) - z^{1}_{x}(x-y-z))\sinh(z^2(x-y) - z^2(x-y-z))(z^{2}_{xx}(x-y) - z^{2}_{xx}(x-y-z))}{(\cosh(z^2(x-y) - z^2(x-y-z)) - \cos(z^1(x-y) - z^1(x-y-z)))^2}
\right)
dy dz 
\end{multline*}

\begin{multline*}
 B_{47}(x) = 2\int_{\T} \int_{\T} \frac{\sin(z^1(x) - z^1(x-y)))}{\cosh(z^2(x) - z^2(x-y)) - \cos(z^1(x) - z^1(x-y))}\\ 
\times \left(\frac{\sin(z^1(x) - z^1(x-z))(z^{1}_{x}(x) - z^{1}_{x}(x-z))\sinh(z^2(x) - z^2(x-z))(z^{2}_{x}(x) - z^{2}_{x}(x-z))}{(\cosh(z^2(x) - z^2(x-z)) - \cos(z^1(x) - z^1(x-z)))^{3}}\right.
\\
\times(\sinh(z^{2}(x) - z^{2}(x-z))(z^{2}_{x}(x)-z^{2}_{x}(x-z)) + \sin(z^{1}(x) - z^{1}(x-z))(z^{1}_{x}(x)-z_{x}^{1}(x-z)))
\\
- 
\frac{\sin(z^1(x-y) - z^1(x-y-z))(z^{1}_{x}(x-y) - z^{1}_{x}(x-y-z))\sinh(z^2(x-y) - z^2(x-y-z))(z^{2}_{x}(x-y) - z^{2}_{x}(x-y-z))}{(\cosh(z^2(x-y) - z^2(x-y-z)) - \cos(z^1(x-y) - z^1(x-y-z)))^{3}} \\
\left.\times
(\sinh(z^{2}(x-y) - z^{2}(x-y-z))(z^{2}_{x}(x-y)-z^{2}_{x}(x-y-z)) + \sin(z^{1}(x-y) - z^{1}(x-y-z))(z^{1}_{x}(x-y)-z_{x}^{1}(x-y-z)))\right)
dy dz 
\end{multline*}

We keep on differentiating, this time $I_{24}(x)$:
\begin{align*}
 \pa_{x} I_{24}(x) & = B_{51}(x) + B_{52}(x) + B_{53}(x) + B_{54}(x) + B_{55}(x),
\end{align*}
which have the following expressions:
\begin{multline*}
 B_{51}(x) = -\int_{\T} \int_{\T} \frac{\cos(z^1(x) - z^1(x-y))(z^1_{x}(x) - z^1_{x}(x-y))}{\cosh(z^2(x) - z^2(x-y)) - \cos(z^1(x) - z^1(x-y))}\\ 
\times \left(\frac{(\sin(z^1(x) - z^1(x-z)))^{2}(z^{1}_{x}(x) - z^{1}_{x}(x-z))^{2}}{(\cosh(z^2(x) - z^2(x-z)) - \cos(z^1(x) - z^1(x-z)))^{2}}\right.
\\
\left. - 
\frac{(\sin(z^1(x-y) - z^1(x-y-z)))^{2}(z^{1}_{x}(x-y) - z^{1}_{x}(x-y-z))^{2}}{(\cosh(z^2(x-y) - z^2(x-y-z)) - \cos(z^1(x-y) - z^1(x-y-z)))^2}
\right)
dy dz 
\end{multline*}

\begin{multline*}
 B_{52}(x) = \int_{\T} \int_{\T} \frac{\sin(z^1(x) - z^1(x-y)))}{(\cosh(z^2(x) - z^2(x-y)) - \cos(z^1(x) - z^1(x-y)))^{2}}\\ 
\times \left(\sinh(z^{2}(x) - z^{2}(x-y))(z^{2}_{x}(x) - z^{2}_{x}(x-y)) + \sin(z^1(x) - z^1(x-y))(z^{1}_{x}(x) - z^{1}_{x}(x-y))\right) \\
\times \left(\frac{(\sin(z^1(x) - z^1(x-z)))^{2}(z^{1}_{x}(x) - z^{1}_{x}(x-z))^{2}}{(\cosh(z^2(x) - z^2(x-z)) - \cos(z^1(x) - z^1(x-z)))^{2}}\right.
\\
\left. - 
\frac{(\sin(z^1(x-y) - z^1(x-y-z)))^{2}(z^{1}_{x}(x-y) - z^{1}_{x}(x-y-z))^{2}}{(\cosh(z^2(x-y) - z^2(x-y-z)) - \cos(z^1(x-y) - z^1(x-y-z)))^2}
\right)
dy dz 
\end{multline*}

\begin{multline*}
 B_{53}(x) = -2\int_{\T} \int_{\T} \frac{\sin(z^1(x) - z^1(x-y)))}{\cosh(z^2(x) - z^2(x-y)) - \cos(z^1(x) - z^1(x-y))}\\ 
\times \left(\frac{\sin(z^1(x) - z^1(x-z))\cos(z^1(x) - z^1(x-z))(z^{1}_{x}(x) - z^{1}_{x}(x-z))^{3}}{(\cosh(z^2(x) - z^2(x-z)) - \cos(z^1(x) - z^1(x-z)))^{2}}\right.
\\
\left. - 
\frac{\sin(z^1(x-y) - z^1(x-y-z))\cos(z^1(x-y) - z^1(x-y-z))(z^{1}_{x}(x-y) - z^{1}_{x}(x-y-z))^{3}}{(\cosh(z^2(x-y) - z^2(x-y-z)) - \cos(z^1(x-y) - z^1(x-y-z)))^2}
\right)
dy dz 
\end{multline*}

\begin{multline*}
 B_{54}(x) = -2\int_{\T} \int_{\T} \frac{\sin(z^1(x) - z^1(x-y)))}{\cosh(z^2(x) - z^2(x-y)) - \cos(z^1(x) - z^1(x-y))}\\ 
\times \left(\frac{(\sin(z^1(x) - z^1(x-z)))^{2}(z^{1}_{x}(x) - z^{1}_{x}(x-z))(z^{1}_{xx}(x) - z^{1}_{xx}(x-z))}{(\cosh(z^2(x) - z^2(x-z)) - \cos(z^1(x) - z^1(x-z)))^{2}}\right.
\\
\left. - 
\frac{(\sin(z^1(x-y) - z^1(x-y-z)))^{2}(z^{1}_{x}(x-y) - z^{1}_{x}(x-y-z))(z^{1}_{xx}(x-y) - z^{1}_{xx}(x-y-z))}{(\cosh(z^2(x-y) - z^2(x-y-z)) - \cos(z^1(x-y) - z^1(x-y-z)))^2}
\right)
dy dz 
\end{multline*}

\begin{multline*}
 B_{55}(x) = 2\int_{\T} \int_{\T} \frac{\sin(z^1(x) - z^1(x-y)))}{\cosh(z^2(x) - z^2(x-y)) - \cos(z^1(x) - z^1(x-y))}\\ 
\times \left(\frac{(\sin(z^1(x) - z^1(x-z)))^{2}(z^{1}_{x}(x) - z^{1}_{x}(x-z))^{2}}{(\cosh(z^2(x) - z^2(x-z)) - \cos(z^1(x) - z^1(x-z)))^{3}}\right.
\\
\times(\sinh(z^{2}(x) - z^{2}(x-z))(z^{2}_{x}(x)-z^{2}_{x}(x-z)) + \sin(z^{1}(x) - z^{1}(x-z))(z^{1}_{x}(x)-z_{x}^{1}(x-z)))
\\
 - 
\frac{(\sin(z^1(x-y) - z^1(x-y-z)))^{2}(z^{1}_{x}(x-y) - z^{1}_{x}(x-y-z))^{2}}{(\cosh(z^2(x-y) - z^2(x-y-z)) - \cos(z^1(x-y) - z^1(x-y-z)))^{3}}
\\
\left.\times
(\sinh(z^{2}(x-y) - z^{2}(x-y-z))(z^{2}_{x}(x-y)-z^{2}_{x}(x-y-z)) + \sin(z^{1}(x-y) - z^{1}(x-y-z))(z^{1}_{x}(x-y)-z_{x}^{1}(x-y-z)))\right)
dy dz 
\end{multline*}

After that, we differentiate $I_{3}(x)$, resulting in:
\begin{align*}
 \pa_{x} I_{3}(x) = B_{61}(x) + B_{62}(x) + B_{63}(x) + B_{64}(x) + B_{65}(x) + B_{66}(x) + B_{67}(x)
\end{align*}
with
\begin{multline*}
 B_{61}(x)  = - \int_{\T} \int_{\T} \frac{\cos(z^1(x) - z^1(x-y))(z^{1}_{x}(x) - z^{1}_{x}(x-y))^{2}\sinh(z^{2}(x) - z^{2}(x-y))}{(\cosh(z^2(x) - z^2(x-y)) - \cos(z^1(x) - z^1(x-y)))^{2}} \\
 \times \left(\frac{\sin(z^1(x) - z^1(x-z))(z^{2}_{x}(x) - z^{2}_{x}(x-z))}{\cosh(z^2(x) - z^2(x-z)) - \cos(z^1(x) - z^1(x-z))}\right.
\\
\left. - 
\frac{\sin(z^1(x-y) - z^1(x-y-z))(z^{2}_{x}(x-y) - z^{2}_{x}(x-y-z))}{\cosh(z^2(x-y) - z^2(x-y-z)) - \cos(z^1(x-y) - z^1(x-y-z))}
\right)
dy dz 
\end{multline*}

\begin{multline*}
 B_{62}(x)  = - \int_{\T} \int_{\T} \frac{\sin(z^1(x) - z^1(x-y))(z^{1}_{xx}(x) - z^{1}_{xx}(x-y))\sinh(z^{2}(x) - z^{2}(x-y))}{(\cosh(z^2(x) - z^2(x-y)) - \cos(z^1(x) - z^1(x-y)))^{2}} \\
 \times \left(\frac{\sin(z^1(x) - z^1(x-z))(z^{2}_{x}(x) - z^{2}_{x}(x-z))}{\cosh(z^2(x) - z^2(x-z)) - \cos(z^1(x) - z^1(x-z))}\right.
\\
\left. - 
\frac{\sin(z^1(x-y) - z^1(x-y-z))(z^{2}_{x}(x-y) - z^{2}_{x}(x-y-z))}{\cosh(z^2(x-y) - z^2(x-y-z)) - \cos(z^1(x-y) - z^1(x-y-z))}
\right)
dy dz 
\end{multline*}

\begin{multline*}
 B_{63}(x)  = - \int_{\T} \int_{\T} \frac{\sin(z^1(x) - z^1(x-y))(z^{1}_{x}(x) - z^{1}_{x}(x-y))\cosh(z^{2}(x) - z^{2}(x-y))(z^{2}_{x}(x) - z^{2}_{x}(x-y))}{(\cosh(z^2(x) - z^2(x-y)) - \cos(z^1(x) - z^1(x-y)))^{2}} \\
 \times \left(\frac{\sin(z^1(x) - z^1(x-z))(z^{2}_{x}(x) - z^{2}_{x}(x-z))}{\cosh(z^2(x) - z^2(x-z)) - \cos(z^1(x) - z^1(x-z))}\right.
\\
\left. - 
\frac{\sin(z^1(x-y) - z^1(x-y-z))(z^{2}_{x}(x-y) - z^{2}_{x}(x-y-z))}{\cosh(z^2(x-y) - z^2(x-y-z)) - \cos(z^1(x-y) - z^1(x-y-z))}
\right)
dy dz 
\end{multline*}

\begin{multline*}
 B_{64}(x)  = 2\int_{\T} \int_{\T} \frac{\sin(z^1(x) - z^1(x-y))(z^{1}_{x}(x) - z^{1}_{x}(x-y))\sinh(z^{2}(x) - z^{2}(x-y))}{(\cosh(z^2(x) - z^2(x-y)) - \cos(z^1(x) - z^1(x-y)))^{3}} \\
\times \left(\sinh(z^{2}(x) - z^{2}(x-y))(z^{2}_{x}(x) - z^{2}_{x}(x-y)) + \sin(z^1(x) - z^1(x-y))(z^{1}_{x}(x) - z^{1}_{x}(x-y))\right) \\
 \times \left(\frac{\sin(z^1(x) - z^1(x-z))(z^{2}_{x}(x) - z^{2}_{x}(x-z))}{\cosh(z^2(x) - z^2(x-z)) - \cos(z^1(x) - z^1(x-z))}\right.
\\
\left. - 
\frac{\sin(z^1(x-y) - z^1(x-y-z))(z^{2}_{x}(x-y) - z^{2}_{x}(x-y-z))}{\cosh(z^2(x-y) - z^2(x-y-z)) - \cos(z^1(x-y) - z^1(x-y-z))}
\right)
dy dz 
\end{multline*}

\begin{multline*}
 B_{65}(x)  = - \int_{\T} \int_{\T} \frac{\sin(z^1(x) - z^1(x-y))(z^{1}_{x}(x) - z^{1}_{x}(x-y))\sinh(z^{2}(x) - z^{2}(x-y))}{(\cosh(z^2(x) - z^2(x-y)) - \cos(z^1(x) - z^1(x-y)))^{2}} \\
 \times \left(\frac{\cos(z^1(x) - z^1(x-z))(z^1_{x}(x) - z^1_{x}(x-z))(z^{2}_{x}(x) - z^{2}_{x}(x-z))}{\cosh(z^2(x) - z^2(x-z)) - \cos(z^1(x) - z^1(x-z))}\right.
\\
\left. - 
\frac{\cos(z^1(x-y) - z^1(x-y-z))(z^1_{x}(x-y) - z^1_{x}(x-y-z))(z^{2}_{x}(x-y) - z^{2}_{x}(x-y-z))}{\cosh(z^2(x-y) - z^2(x-y-z)) - \cos(z^1(x-y) - z^1(x-y-z))}
\right)
dy dz 
\end{multline*}

\begin{multline*}
 B_{66}(x)  = - \int_{\T} \int_{\T} \frac{\sin(z^1(x) - z^1(x-y))(z^{1}_{x}(x) - z^{1}_{x}(x-y))\sinh(z^{2}(x) - z^{2}(x-y))}{(\cosh(z^2(x) - z^2(x-y)) - \cos(z^1(x) - z^1(x-y)))^{2}} \\
 \times \left(\frac{\sin(z^1(x) - z^1(x-z))(z^{2}_{xx}(x) - z^{2}_{xx}(x-z))}{\cosh(z^2(x) - z^2(x-z)) - \cos(z^1(x) - z^1(x-z))}\right.
\\
\left. - 
\frac{\sin(z^1(x-y) - z^1(x-y-z))(z^{2}_{xx}(x-y) - z^{2}_{xx}(x-y-z))}{\cosh(z^2(x-y) - z^2(x-y-z)) - \cos(z^1(x-y) - z^1(x-y-z))}
\right)
dy dz 
\end{multline*}

\begin{multline*}
 B_{67}(x)  = \int_{\T} \int_{\T} \frac{\sin(z^1(x) - z^1(x-y))(z^{1}_{x}(x) - z^{1}_{x}(x-y))\sinh(z^{2}(x) - z^{2}(x-y))}{(\cosh(z^2(x) - z^2(x-y)) - \cos(z^1(x) - z^1(x-y)))^{2}} \\
 \times \left(\frac{\sin(z^1(x) - z^1(x-z))(z^{2}_{x}(x) - z^{2}_{x}(x-z))}{(\cosh(z^2(x) - z^2(x-z)) - \cos(z^1(x) - z^1(x-z)))^{2}}\right.
\\
\times(\sinh(z^{2}(x) - z^{2}(x-z))(z^{2}_{x}(x)-z^{2}_{x}(x-z)) + \sin(z^{1}(x) - z^{1}(x-z))(z^{1}_{x}(x)-z_{x}^{1}(x-z)))
\\ - 
\frac{\sin(z^1(x-y) - z^1(x-y-z))(z^{2}_{x}(x-y) - z^{2}_{x}(x-y-z))}{(\cosh(z^2(x-y) - z^2(x-y-z)) - \cos(z^1(x-y) - z^1(x-y-z)))^{2}}
\\
\left.\times
(\sinh(z^{2}(x-y) - z^{2}(x-y-z))(z^{2}_{x}(x-y)-z^{2}_{x}(x-y-z)) + \sin(z^{1}(x-y) - z^{1}(x-y-z))(z^{1}_{x}(x-y)-z_{x}^{1}(x-y-z)))\right)
dy dz 
\end{multline*}

The last term we differentiate is $I_{4}(x)$, which yields
\begin{align*}
 \pa_{x}I_{4}(x) = B_{71}(x) + B_{72}(x) + B_{73}(x) + B_{74}(x) + B_{75}(x) + B_{76}(x)
\end{align*}

\begin{multline*}
 B_{71}(x)  = - 2\int_{\T} \int_{\T} \frac{(\sin(z^1(x) - z^1(x-y)))(\cos(z^1(x) - z^1(x-y)))(z^{1}_{x}(x) - z^{1}_{x}(x-y))^{2}}{(\cosh(z^2(x) - z^2(x-y)) - \cos(z^1(x) - z^1(x-y)))^{2}} \\
 \times \left(\frac{\sin(z^1(x) - z^1(x-z))(z^{1}_{x}(x) - z^{1}_{x}(x-z))}{\cosh(z^2(x) - z^2(x-z)) - \cos(z^1(x) - z^1(x-z))}\right.
\\
\left. - 
\frac{\sin(z^1(x-y) - z^1(x-y-z))(z^{1}_{x}(x-y) - z^{1}_{x}(x-y-z))}{\cosh(z^2(x-y) - z^2(x-y-z)) - \cos(z^1(x-y) - z^1(x-y-z))}
\right)
dy dz 
\end{multline*}

\begin{multline*}
 B_{72}(x)  = - \int_{\T} \int_{\T} \frac{(\sin(z^1(x) - z^1(x-y)))^{2}(z^{1}_{xx}(x) - z^{1}_{xx}(x-y))}{(\cosh(z^2(x) - z^2(x-y)) - \cos(z^1(x) - z^1(x-y)))^{2}} \\
 \times \left(\frac{\sin(z^1(x) - z^1(x-z))(z^{1}_{x}(x) - z^{1}_{x}(x-z))}{\cosh(z^2(x) - z^2(x-z)) - \cos(z^1(x) - z^1(x-z))}\right.
\\
\left. - 
\frac{\sin(z^1(x-y) - z^1(x-y-z))(z^{1}_{x}(x-y) - z^{1}_{x}(x-y-z))}{\cosh(z^2(x-y) - z^2(x-y-z)) - \cos(z^1(x-y) - z^1(x-y-z))}
\right)
dy dz 
\end{multline*}

\begin{multline*}
 B_{73}(x)  = 2\int_{\T} \int_{\T} \frac{(\sin(z^1(x) - z^1(x-y)))^{2}(z^{1}_{x}(x) - z^{1}_{x}(x-y))}{(\cosh(z^2(x) - z^2(x-y)) - \cos(z^1(x) - z^1(x-y)))^{3}} \\
\times \left(\sinh(z^{2}(x) - z^{2}(x-y))(z^{2}_{x}(x) - z^{2}_{x}(x-y)) + \sin(z^1(x) - z^1(x-y))(z^{1}_{x}(x) - z^{1}_{x}(x-y))\right) \\
 \times \left(\frac{\sin(z^1(x) - z^1(x-z))(z^{1}_{x}(x) - z^{1}_{x}(x-z))}{\cosh(z^2(x) - z^2(x-z)) - \cos(z^1(x) - z^1(x-z))}\right.
\\
\left. - 
\frac{\sin(z^1(x-y) - z^1(x-y-z))(z^{1}_{x}(x-y) - z^{1}_{x}(x-y-z))}{\cosh(z^2(x-y) - z^2(x-y-z)) - \cos(z^1(x-y) - z^1(x-y-z))}
\right)
dy dz 
\end{multline*}

\begin{multline*}
 B_{74}(x)  = - \int_{\T} \int_{\T} \frac{(\sin(z^1(x) - z^1(x-y)))^{2}(z^{1}_{x}(x) - z^{1}_{x}(x-y))}{(\cosh(z^2(x) - z^2(x-y)) - \cos(z^1(x) - z^1(x-y)))^{2}} \\
 \times \left(\frac{\cos(z^1(x) - z^1(x-z))(z^{1}_{x}(x) - z^{1}_{x}(x-z))^{2}}{\cosh(z^2(x) - z^2(x-z)) - \cos(z^1(x) - z^1(x-z))}\right.
\\
\left. - 
\frac{\cos(z^1(x-y) - z^1(x-y-z))(z^{1}_{x}(x-y) - z^{1}_{x}(x-y-z))^{2}}{\cosh(z^2(x-y) - z^2(x-y-z)) - \cos(z^1(x-y) - z^1(x-y-z))}
\right)
dy dz 
\end{multline*}

\begin{multline*}
 B_{75}(x)  = - \int_{\T} \int_{\T} \frac{(\sin(z^1(x) - z^1(x-y)))^{2}(z^{1}_{x}(x) - z^{1}_{x}(x-y))}{(\cosh(z^2(x) - z^2(x-y)) - \cos(z^1(x) - z^1(x-y)))^{2}} \\
 \times \left(\frac{\sin(z^1(x) - z^1(x-z))(z^{1}_{xx}(x) - z^{1}_{xx}(x-z))}{\cosh(z^2(x) - z^2(x-z)) - \cos(z^1(x) - z^1(x-z))}\right.
\\
\left. - 
\frac{\sin(z^1(x-y) - z^1(x-y-z))(z^{1}_{xx}(x-y) - z^{1}_{xx}(x-y-z))}{\cosh(z^2(x-y) - z^2(x-y-z)) - \cos(z^1(x-y) - z^1(x-y-z))}
\right)
dy dz 
\end{multline*}

\begin{multline*}
 B_{76}(x)  = \int_{\T} \int_{\T} \frac{(\sin(z^1(x) - z^1(x-y)))^{2}(z^{1}_{x}(x) - z^{1}_{x}(x-y))}{(\cosh(z^2(x) - z^2(x-y)) - \cos(z^1(x) - z^1(x-y)))^{2}} \\
 \times \left(\frac{\sin(z^1(x) - z^1(x-z))(z^{1}_{x}(x) - z^{1}_{x}(x-z))}{(\cosh(z^2(x) - z^2(x-z)) - \cos(z^1(x) - z^1(x-z)))^{2}}\right.
\\
\times(\sinh(z^{2}(x) - z^{2}(x-z))(z^{2}_{x}(x)-z^{2}_{x}(x-z)) + \sin(z^{1}(x) - z^{1}(x-z))(z^{1}_{x}(x)-z_{x}^{1}(x-z)))
\\- 
\frac{\sin(z^1(x-y) - z^1(x-y-z))(z^{1}_{x}(x-y) - z^{1}_{x}(x-y-z))}{(\cosh(z^2(x-y) - z^2(x-y-z)) - \cos(z^1(x-y) - z^1(x-y-z)))^{2}} \\
\left.\times
(\sinh(z^{2}(x-y) - z^{2}(x-y-z))(z^{2}_{x}(x-y)-z^{2}_{x}(x-y-z)) + \sin(z^{1}(x-y) - z^{1}(x-y-z))(z^{1}_{x}(x-y)-z_{x}^{1}(x-y-z)))\right)
dy dz 
\end{multline*}

\section{Auxiliary tables}

\begin{longtable}{|c|c|c|c|c|}
\hline
Integral & Degree Num. in $y$ & Degree Num. in $z$ & Degree Den. in $y$ & Degree Den. in $z$ \\
\hline
$ B_{11}$ &  6 & 4 & 2 & 4 \\
\hline
$ B_{12}$ & 2 & 4 & 2 & 4\\
\hline
$ B_{13}$ & 6 & 4 & 4 & 4\\
\hline
$ B_{14}$ & 3 & 4 & 2 & 4\\
\hline
$ B_{15}$ & 3 & 4 & 2 & 4\\
\hline
$ B_{16}$ & 3 & 8 & 2 & 8\\
\hline
$ B_{21}$ & 3 & 4 & 2 & 4\\
\hline
$ B_{22}$ & 5 & 4 & 4 & 4\\
\hline
$ B_{23}$ & 2 & 6 & 2 & 4\\
\hline
$ B_{24}$ & 2 & 4 & 2 & 4\\
\hline
$ B_{25}$ & 2 & 8 & 2 & 8\\
\hline
$ B_{31}$ & 3 & 4 & 2 & 4\\
\hline
$ B_{32}$ & 5 & 4 & 4 & 4\\
\hline
$ B_{33}$ & 2 & 4 & 2 & 4\\
\hline
$ B_{34}$ & 2 & 4 & 2 & 4\\
\hline
$ B_{35}$ & 2 & 8 & 2 & 8\\
\hline
$ B_{41}$ & 3 & 8 & 2 & 8\\
\hline
$ B_{42}$ & 5 & 8 & 4 & 8\\
\hline
$ B_{43}$ & 2 & 8 & 2 & 8\\
\hline
$ B_{44}$ & 2 & 8 & 2 & 8\\
\hline
$ B_{45}$ & 2 & 8 & 2 & 8\\
\hline
$ B_{46}$ & 2 & 8 & 2 & 8\\
\hline
$ B_{47}$ & 2 & 12 & 2 & 12\\
\hline
$ B_{51}$ & 3 & 8 & 2 & 8\\
\hline
$ B_{52}$ & 5 & 8 & 4 & 8\\
\hline
$ B_{53}$ & 2 & 8 & 2 & 8\\
\hline
$ B_{54}$ & 2 & 8 & 2 & 8\\
\hline
$ B_{55}$ & 2 & 12 & 2 & 12\\
\hline
$ B_{61}$ & 6 & 4 & 4 & 4\\
\hline
$ B_{62}$ & 4 & 4 & 4 & 4\\
\hline
$ B_{63}$ & 6 & 4 & 4 & 4\\
\hline
$ B_{64}$ & 8 & 4 & 6 & 4\\
\hline
$ B_{65}$ & 5 & 4 & 4 & 4\\
\hline
$ B_{66}$ & 5 & 4 & 4 & 4\\
\hline
$ B_{67}$ & 5 & 8 & 4 & 8\\
\hline
$ B_{71}$ & 6 & 4 & 4 & 4\\
\hline
$ B_{72}$ & 4 & 4 & 4 & 4\\
\hline
$ B_{73}$ & 8 & 4 & 6 & 4\\
\hline
$ B_{74}$ & 5 & 4 & 4 & 4\\
\hline
$ B_{75}$ & 5 & 4 & 4 & 4\\
\hline
$ B_{76}$ & 5 & 8 & 4 & 8\\
\hline
\caption{Degree of the Taylor expansions in $y$ and $z$ of the different integrands written down as fractions $\frac{\text{numerator}}{\text{denominator}}$.}
\label{table_taylor}
\end{longtable}

\begin{longtable}{|c|c|c|c|c|}
\hline
Integral & Bounded Region & Singularity Center & Singularity $y$ Axis & Singularity $z$ Axis \\
\hline
$ B_{11}$ & $-21.93^{58}_{09}$ & $[-4.9 \cdot 10^{-13}, 4.9 \cdot 10^{-13}]$ & $[-6.3 \cdot 10^{-7}, 6.3 \cdot 10^{-7}]$ & $-0.20^{71}_{50}$ \\
\hline
$ B_{12}$ & $19.1^{19}_{32}$ & $[8.9 \cdot 10^{-8}, 3.7 \cdot 10^{-7}]$ & $[-2.3 \cdot 10^{-3}, 2.5 \cdot 10^{-3}]$ & $0.2^{76}_{82}$\\
\hline
$ B_{13}$ & $-2.^{13}_{03}$ & $[-4.2 \cdot 10^{-8}, 5.6 \cdot 10^{-8}]$ & $[-3.2 \cdot 10^{-5}, 3.2 \cdot 10^{-5}]$ & $0.^{38}_{41}$\\
\hline
$ B_{14}$ & $4.^{39}_{41}$ & $[7.8 \cdot 10^{-8}, 1.9 \cdot 10^{-7}]$ & $[-1.1 \cdot 10^{-4}, 1.0 \cdot 10^{-4}]$ & $0.2^{76}_{82}$\\
\hline
$ B_{15}$ & $8.5^{35}_{43}$ & $[-9.0 \cdot 10^{-10}, 1.3 \cdot 10^{-7}]$ & $[-7.7 \cdot 10^{-5}, 6.1 \cdot 10^{-5}]$ & $0.0^{88}_{91}$\\
\hline
$ B_{16}$ & $1^{4.9}_{5.1}$ & $[-7.3 \cdot 10^{-8}, 7.4 \cdot 10^{-8}]$ & $[-7.0 \cdot 10^{-4}, 8.8 \cdot 10^{-4}]$ & $-0.^{42}_{28}$\\
\hline
$ B_{21}$ & $4.^{39}_{41}$ &  $[7.8 \cdot 10^{-8}, 1.9 \cdot 10^{-7}]$ & $[-1.1 \cdot 10^{-4}, 1.0 \cdot 10^{-4}]$ & $0.2^{76}_{82}$\\
\hline
$ B_{22}$ & $14.^{55}_{62}$ & $[-6.0 \cdot 10^{-8}, 6.0 \cdot 10^{-8}]$ & $[-5.9 \cdot 10^{-6}, 5.9 \cdot 10^{-6}]$ & $0.1^{77}_{91}$\\
\hline
$ B_{23}$ & $-13.35^{59}_{12}$ & $[-9.2 \cdot 10^{-10}, 9.2 \cdot 10^{-10}]$ & $[-4.1 \cdot 10^{-5}, 6.6 \cdot 10^{-5}]$ & $-0.0000^{46}_{33}$\\
\hline
$ B_{24}$ & $3^{0.9}_{1.1}$ & $[-2.5 \cdot 10^{-8}, 1.9 \cdot 10^{-7}]$ & $[-1.5 \cdot 10^{-4}, 1.9 \cdot 10^{-4}]$ & $0.04^{22}_{52}$\\
\hline
$ B_{25}$ & $-29.^{79}_{66}$ & $[-3.5 \cdot 10^{-7}, 3.5 \cdot 10^{-7}]$ & $[-1.7 \cdot 10^{-3}, 1.6 \cdot 10^{-3}]$ & $-0.^{28}_{19}$\\
\hline
$ B_{31}$ & $8.5^{35}_{43}$ & $[-9.0 \cdot 10^{-10}, 1.3 \cdot 10^{-7}]$ & $[-7.7 \cdot 10^{-5}, 6.1 \cdot 10^{-5}]$ & $0.0^{88}_{91}$\\
\hline
$ B_{32}$ & $10.^{09}_{14}$ & $[-6.2 \cdot 10^{-8}, 6.3 \cdot 10^{-8}]$ & $[-5.8 \cdot 10^{-6}, 5.8 \cdot 10^{-6}]$ & $0.28^{21}_{74}$\\
\hline
$ B_{33}$ & $15.^{49}_{51}$ & $[-3.6 \cdot 10^{-8}, 1.2 \cdot 10^{-7}]$ & $[-7.2 \cdot 10^{-5}, 9.5 \cdot 10^{-5}]$ & $0.02^{11}_{26}$\\
\hline
$ B_{34}$ & $-14.94^{63}_{18}$ & $[-5.5 \cdot 10^{-8}, 5.4 \cdot 10^{-8}]$ & $[-5.0 \cdot 10^{-5}, 4.9 \cdot 10^{-5}]$ & $-0.14^{52}_{36}$\\
\hline
$ B_{35}$ & $-9.^{96}_{88}$ & $[-6.3 \cdot 10^{-7}, 6.3 \cdot 10^{-7}]$ & $[-1.2 \cdot 10^{-3}, 1.2 \cdot 10^{-3}]$ & $-0.0^{87}_{48}$\\
\hline
$ B_{41}$ & $-6.^{62}_{47}$ & $[-7.3 \cdot 10^{-8}, 7.4 \cdot 10^{-8}]$ & $[-7.2 \cdot 10^{-4}, 8.0 \cdot 10^{-4}]$ & $-0.^{21}_{08}$\\
\hline
$ B_{42}$ & $15.^{13}_{60}$ & $[-5.5 \cdot 10^{-8}, 5.6 \cdot 10^{-8}]$ & $[-2.5 \cdot 10^{-5}, 2.5 \cdot 10^{-5}]$ & $-1.^{40}_{17}$\\
\hline
$ B_{43}$ & $-0.^{64}_{52}$ & $[-3.5 \cdot 10^{-7}, 3.5 \cdot 10^{-7}]$ & $[-1.7 \cdot 10^{-3}, 1.7 \cdot 10^{-3}]$ & $-0.^{13}_{05}$\\
\hline
$ B_{44}$ & $9.^{07}_{15}$ & $[-6.2 \cdot 10^{-7}, 6.3 \cdot 10^{-7}]$ & $[-1.3 \cdot 10^{-3}, 1.2 \cdot 10^{-3}]$ & $-0.0^{73}_{36}$\\
\hline
$ B_{45}$ & $-751.^{73}_{06}$ & $[-3.0 \cdot 10^{-7}, 3.0 \cdot 10^{-7}]$ & $[-1.4 \cdot 10^{-2}, 1.2 \cdot 10^{-2}]$ & $[-4.2,-3.8]$\\
\hline
$ B_{46}$ & $50.^{27}_{33}$ & $[-6.9 \cdot 10^{-7}, 6.9 \cdot 10^{-7}]$ & $[-5.3 \cdot 10^{-4}, 6.1 \cdot 10^{-4}]$ & $0.6^{53}_{99}$\\
\hline
$ B_{47}$ & $68^{5.1}_{7.2}$ & $[-2.1 \cdot 10^{-6}, 2.1 \cdot 10^{-6}]$ & $[-2.9 \cdot 10^{-2}, 3.2 \cdot 10^{-2}]$ & $[3.7,4.7]$\\
\hline
$ B_{51}$ & $21.5^{63}_{84}$ & $[-6.8 \cdot 10^{-8}, 8.0 \cdot 10^{-8}]$ & $[-2.1 \cdot 10^{-4}, 3.1 \cdot 10^{-4}]$ & $-0.^{22}_{19}$\\
\hline
$ B_{52}$ & $-15.^{13}_{02}$ & $[-8.9 \cdot 10^{-8}, 8.8 \cdot 10^{-8}]$ & $[-1.1 \cdot 10^{-5}, 1.1 \cdot 10^{-5}]$ & $-0.0^{77}_{39}$\\
\hline
$ B_{53}$ & $-58.^{31}_{27}$ & $[-5.1 \cdot 10^{-8}, 5.1 \cdot 10^{-8}]$ & $[-1.2 \cdot 10^{-3}, 9.7 \cdot 10^{-4}]$ & $-0.^{31}_{28}$\\
\hline
$ B_{54}$ & $-38.0^{75}_{55}$ & $[-5.5 \cdot 10^{-8}, 5.5 \cdot 10^{-8}]$ & $[-6.8 \cdot 10^{-4}, 6.7 \cdot 10^{-4}]$ & $-0.0^{32}_{19}$\\
\hline
$ B_{55}$ & $10^{3.5}_{4.1}$ & $[-5.1 \cdot 10^{-8}, 5.1 \cdot 10^{-8}]$ & $[-8.3 \cdot 10^{-3}, 8.5 \cdot 10^{-3}]$ & $0.^{34}_{52}$\\
\hline
$ B_{61}$ & $48.5^{07}_{82}$ & $[-4.9 \cdot 10^{-8}, 6.3 \cdot 10^{-8}]$ & $[2.5 \cdot 10^{-5}, 4.1 \cdot 10^{-5}]$ & $0.5^{09}_{27}$\\
\hline
$ B_{62}$ & $33.6^{46}_{95}$ & $[-2.2 \cdot 10^{-7}, 2.2 \cdot 10^{-7}]$ & $[-5.0 \cdot 10^{-4}, 5.4 \cdot 10^{-4}]$ & $0.4^{77}_{88}$\\
\hline
$ B_{63}$ & $-61^{5.3}_{4.7}$ & $[-4.7 \cdot 10^{-8}, 4.7 \cdot 10^{-8}]$ & $[-5.4 \cdot 10^{-6}, 5.4 \cdot 10^{-6}]$ & $0.4^{18}_{96}$\\
\hline
$ B_{64}$ & $49^{4.9}_{5.9}$ & $[-1.1 \cdot 10^{-7}, 1.1 \cdot 10^{-7}]$ & $[-4.7 \cdot 10^{-6}, 4.7 \cdot 10^{-6}]$ & $-0.^{54}_{41}$\\
\hline
$ B_{65}$ & $49.^{09}_{16}$ & $[-6.1 \cdot 10^{-8}, 6.1 \cdot 10^{-8}]$ & $[-4.9 \cdot 10^{-6}, 4.9 \cdot 10^{-6}]$ & $0.4^{77}_{88}$\\
\hline
$ B_{66}$ & $7.4^{18}_{57}$ & $[-6.2 \cdot 10^{-8}, 6.2 \cdot 10^{-8}]$ & $[-4.7 \cdot 10^{-6}, 4.7 \cdot 10^{-6}]$ & $-0.0^{85}_{63}$\\
\hline
$ B_{67}$ & $-24.^{62}_{15}$ & $[-5.7 \cdot 10^{-8}, 5.7 \cdot 10^{-8}]$ & $[-1.1 \cdot 10^{-5}, 1.1 \cdot 10^{-5}]$ & $-0.^{68}_{48}$\\
\hline
$ B_{71}$ & $-84.^{81}_{78}$ & $[-8.1 \cdot 10^{-10}, 8.0 \cdot 10^{-10}]$ & $[-3.5 \cdot 10^{-6}, 3.6 \cdot 10^{-6}]$ & $-0.62^{79}_{12}$\\
\hline
$ B_{72}$ & $-29.3^{69}_{59}$ & $[-1.4 \cdot 10^{-7}, 7.0 \cdot 10^{-8}]$ & $[-4.6 \cdot 10^{-6}, 4.6 \cdot 10^{-6}]$ & $-0.22^{39}_{18}$\\
\hline
$ B_{73}$ & $13^{7.8}_{8.1}$ & $[-8.8 \cdot 10^{-10}, 5.4 \cdot 10^{-10}]$ & $[-3.0 \cdot 10^{-6}, 2.8 \cdot 10^{-6}]$ & $0.8^{57}_{94}$\\
\hline
$ B_{74}$ & $7.0^{85}_{96}$ & $[-2.1 \cdot 10^{-9}, 2.4 \cdot 10^{-9}]$ & $[-3.5 \cdot 10^{-6}, 3.3 \cdot 10^{-6}]$ & $-0.22^{39}_{17}$\\
\hline
$ B_{75}$ & $-2.8^{23}_{14}$ & $[-4.2 \cdot 10^{-9}, 4.5 \cdot 10^{-9}]$ & $[-3.9 \cdot 10^{-6}, 3.7 \cdot 10^{-6}]$& $0.04^{29}_{29}$\\
\hline
$ B_{76}$ & $-35.^{63}_{52}$ & $[-5.0 \cdot 10^{-8}, 5.2 \cdot 10^{-8}]$ & $[-3.6 \cdot 10^{-6}, 3.6 \cdot 10^{-6}]$& $0.0^{42}_{68}$\\
\hline
\caption{Detailed breakdown of the rigorous integration results.}
\label{table_results}
\end{longtable}

\begin{longtable}{|c|c|c|}
\hline
Term and region & Number of integrals & Time (HH:MM) \\
\hline
$B_{11}$-$B_{76}$ (nonsingular) & 82 & 14:48\\
\hline
$B_{11}$-$B_{76}$ (center-singular) & 82 & 02:03 \\
\hline
$B_{11}$-$B_{76}$ (singular-first) & 82 & 01:26 \\
\hline
$B_{11}$-$B_{16}$ (singular-second) & 12 & 11:57\\
\hline
$B_{21}$-$B_{25}$ (singular-second) & 10 & 09:57\\
\hline
$B_{31}$-$B_{35}$ (singular-second) & 10 & 11:29\\
\hline
$B_{41}$-$B_{46}$ (singular-second) & 12 & 32:19\\
\hline
$B_{51}$-$B_{54}$ (singular-second) & 8 & 16:44\\
\hline
$B_{61}$-$B_{67}$ (singular-second) & 14 & 13:59\\
\hline
$B_{71}$-$B_{76}$ (singular-second) & 12 & 09:46 \\
\hline
$B_{47}$ (singular-second - subregions 1 and 2) & 4 & 35:53\\
\hline
$B_{47}$ (singular-second - subregions 3 and 4) & 4 & 60:48\\
\hline
$B_{47}$ (singular-second - subregions 5 and 6) & 4 & 82:02\\
\hline
$B_{55}$ (singular-second - subregions 1 and 2) & 4 & 16:02\\
\hline
$B_{55}$ (singular-second - subregions 3 and 4) & 4 & 56:12\\
\hline
$B_{55}$ (singular-second - subregions 5 and 6) & 4 & 74:50\\
\hline
\caption{Performance of the code in the different integrals and regions.}
\label{tableruntime}
\end{longtable}

 \bibliographystyle{abbrv}
 \bibliography{references}

\begin{tabular}{l}
\textbf{Diego C\'ordoba} \\
  {\small Instituto de Ciencias Matem\'aticas} \\
 {\small Consejo Superior de Investigaciones Cient\'ificas} \\
 {\small C/ Nicolas Cabrera, 13-15, 28049 Madrid, Spain} \\
  {\small Email: dcg@icmat.es} \\
\\
\textbf{Javier G\'omez-Serrano} \\
{\small Department of Mathematics} \\
{\small Princeton University}\\
{\small 610 Fine Hall, Washington Rd,}\\
{\small Princeton, NJ 08544, USA}\\
 {\small Email: jg27@math.princeton.edu} \\
    \\
\textbf{Andrej Zlato\v{s}} \\
{\small Department of Mathematics}\\
{\small University of Wisconsin}\\
{\small 480 Lincoln Dr}\\
{\small Madison, WI 53706, USA}\\
{\small Email: andrej@math.wisc.edu}\\
    \\

\end{tabular}

\end{document}